\documentclass[oneside,a4paper,limits]{amsart}

\usepackage[english]{babel}
\usepackage[latin1]{inputenc}
\usepackage[T1]{fontenc}
\usepackage{graphicx}
\usepackage{latexsym}
\usepackage{tensor}
\usepackage{amsmath,amssymb}
\usepackage{ifpdf}
\usepackage{amsthm}
\usepackage{enumitem}
\usepackage{xcolor}

\usepackage{subfig}
\usepackage{graphicx}
\usepackage{pgfplots}
\usetikzlibrary{patterns}
\usepgfplotslibrary{fillbetween}

\swapnumbers

\newtheorem{Theorem}{Theorem}[section]
\newtheorem{prop}[Theorem]{Proposition}
\newtheorem{lem}[Theorem]{Lemma}
\newtheorem{cor}[Theorem]{Corollary}

\theoremstyle{definition}
\newtheorem{rem}[Theorem]{Remark}
\newtheorem{Definition}[Theorem]{Definition}

\newcommand{\R}{\mathbb{R}}
\newcommand{\C}{\mathbb{C}}
\newcommand{\Q}{\mathbb{Q}}
\newcommand{\N}{\mathbb{N}}

\newcommand{\dd}{\mathrm{d}}

\DeclareMathOperator{\supp}{supp}

\DeclareMathOperator{\Tr}{tr}

\DeclareMathOperator{\diag}{diag}

\DeclareMathOperator{\di}{div}
\newcommand{\seb}[1]{{#1}}
\newcommand{\Lcal}{\mathcal{E}'}
\newcommand{\Ecal}{\mathcal{E}}
\newcommand{\skp}[2]{\left\langle {#1},{#2}\right\rangle}
\newcommand{\norm}[1]{\|#1\|}
\newcommand{\normlr}[1]{\left\|#1\right\|}
\newcommand{\normn}[1]{\|#1\|}
\newcommand{\normb}[1]{\bigl\|#1\bigr\|}
\newcommand{\normB}[1]{\Bigl\|#1\Bigr\|}
\newcommand{\normBB}[1]{\biggl\|#1\biggr\|}
\newcommand{\normBBB}[1]{\Biggl\|#1\Biggr\|}
\newcommand{\abs}[1]{|#1|}
\newcommand{\nablasym}{{\bf \varepsilon}}

\newcommand{\mol}{{\hat{v}_{2,\delta}}}

\title[Weak-strong uniqueness for fluid-structure interactions]{Weak-strong uniqueness for an elastic plate interacting with the Navier Stokes equation
}
\author{Sebastian Schwarzacher and Matthias Sroczinski} 

\address{Department of Mathematics and Physics, Charles University}

\numberwithin{equation}{section}
\begin{document}
	\maketitle
	\begin{abstract}
We show weak-strong uniqueness and stability results for the motion of a two or three dimensional fluid governed by the Navier-Stokes equation interacting with a flexible, elastic plate of Koiter type. The plate is situated at the top of the fluid and as such determines the variable part of a time changing domain (that is hence a part of the solution) containing the fluid. The uniqueness result is a consequence of a stability estimate where the difference of two solutions is estimated by the distance of the initial values and outer forces. For that we introduce a methodology that overcomes the problem that the two (variable in time) domains of the fluid velocities and pressures are not the same. The estimate holds under the assumption that one of the two weak solutions possesses some additional higher regularity. The additional regularity is exclusively requested for the velocity of one of the solutions resembling the celebrated Ladyzhenskaya-Prodi-Serrin conditions in the given framework. 
\vspace{4pt}



\vspace{4pt}

\noindent\textsc{Date:} \today{}. 
\end{abstract}

\section{Introduction}

\noindent
The paper investigates the interaction between an elastic solid plate and a viscous incompressible fluid. 
%
For the fluid we will consider the three (or two) dimensional {\em Navier-Stokes equations}~\cite{GB1,Ler34}. For the solid we consider a shell or a plate that is modeled as a thin object of one dimension less than the fluid and which is assumed to be fixed on the top of a container (See Figure~1). For  modeling on  {\em elastic plates} see \cite{CiarletBook2,CiarletBook3} and the references therein. The fluid and the plate interact via a kinematic and a dynamic coupling condition on the moving interface.

Our main result consists in the {\em weak-strong uniqueness} of solutions for a flow in a variable 3D (or 2D) domain interacting with a 2D (or 1D) plate (see Theorem~\ref{the:1}). While the regularity of the weak solutions that we use are known to be satisfied for all weak solutions we assume additional regularity of the {\em velocity} of the strong solution, that can be related (via its index) to the celebrated Ladyzhenskaya-Prodi-Serrin conditions~\cite{Pro59,Ser62,Ser63,Lad67,Sve03}. These are conditions for solutions to Navier-Stokes equations in a fixed domain that imply their smoothness and uniqueness.

Please observe, that we do not {assume} any additional regularity of the solid displacement; in particular the domain of the strong fluid-velocity is not even assumed to be uniformly Lipschitz continuous. In order to handle the limited regularity assumptions (on the strong solution) rather complex estimates where necessary. Some of them depend sensitively on {\em a-priori estimates} for the solid deformation shown in~\cite{MuhSch19}. 

To measure the distance between two solutions it is necessary to introduce a change of variables as the domains of the two velocity fields depend on the solution itself. Moreover, since the solid deformation is governed by a hyperbolic equation a mollification in time is unavoidable. In this paper a methodology is introduced that overcomes both obstacles with operators that conserve the property of solenoidality (see Lemma~\ref{lem:molly}).



While the existence theory for weak solutions describing flexible (thin) shells interacting with fluids has been flourishing in the past years \cite{DE,DEGLT,Bou07,CG,Gal,BorSun13,LenRuz14,Len14,BorSun15,
SunBorMulti,muha2016existence,GraHil16,BreSch18,MuhSch19} the uniqueness and stability questions are rather untouched. The only available result for an elastic plate seems to be the work of~\cite{Gui10}; it treats a 1D elastic beam interacting with a 2D fluid whith slip-boundary conditions at the interface.\footnote{Actually some conditions in~\cite{Gui10} could be missing, as the estimate in formula (6.33) on page 25 seems sensitively incorrect. The estimate would only be correct if the distributional time-derivative was in the dual of a Sobolev space and not merely in the dual of its solenoidal subspace.} Otherwise, the only weak-strong uniqueness results for fluid-structure interactions are for non-elastic solids, namely rigid objects~\cite{Sta05,GlaSue15,CheMuhNec17,Bra19}. For fluid-structure interactions involving elastic materials there are some existence results where the uniqueness of strong solutions (in the class of strong solutions) is inherited from the methodology of existence.   
These are short time uniqueness results for strong solutions~\cite{CSS1,CSS2,BdV1,BouGue10comp,GraHil19}, global uniqueness results of strong solutions for small data~\cite{ChuLasWeb13,IgnKukLasTuf17} and the uniqueness for arbitrary times of strong solutions for a 1D visco-elastic plate interacting with a 2D fluid~\cite{GraHil16}. As a consequence of our estimates all constructed strong solutions (involving elastic plates) are unique within the class of {\em weak solutions}.

The {\em applications} within this framework consist in fluids interacting with various thin materials. Of particular interest are those in medicine and biology for arteries or the trachea~\cite{Quarteroni2000,FSIforBIO,hos04}. These fields relay strongly on robust computer simulations, many of which are built along the concept of {\em weak solutions}~\cite{heil2008solvers,helena,Ric17}. Stability results as the one presented here are very suitable to be adapted to such {\em numerical approximations}. We plan to perform that in a future paper.

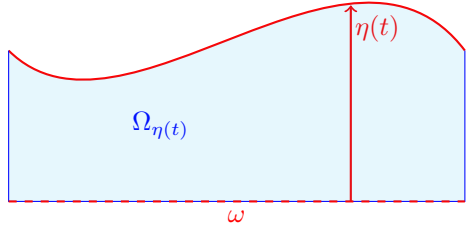
\begin{figure}[!h]\centering
\vspace{-1cm}
\begin{tikzpicture}[scale=1.0]
\draw[thick,red]   (8, 2) .. controls (9.5, 0.5) and (12.5, 4) .. (14, 2);
\draw[blue](8,2)--(8,0)--(14,0)--(14,2);
\draw[dashed, thick, red] (8,0)--(14,0);
\path node at (10,1) {\textcolor{blue}{$\Omega_{\eta(t)}$}};
\path node at (11,-0.2) {\textcolor{red}{\bf $\omega$}};
\draw[->,thick, red] (12.5,0)--(12.5,2.6);
\path node at (12.85,2.3) {\textcolor{red}{\bf $\eta(t)$}};
\fill[color=cyan,opacity=.1]  (8, 2) .. controls (9.5, 0.5) and (12.5, 4) .. (14, 2) -- (14,0) -- (8,0) -- (8,2);
\end{tikzpicture}
\caption{1D plate interacting with a 2D fluid}\label{fig_ALE}
\end{figure}

\subsection{Formulation of the problem}
We consider a $3D$ container whose top wall consist of a $2D$ Koiter type plate (or a $2D$ container whose walls consist of a $1D$ Koiter type plate). As is common for the analysis on plates we assume that the plate can move only upwards and downwards. The deformation of the plate is described by a bounded function $\eta: [0,T] \times {\omega} \to (\delta,\infty)$ for some time interval $[0,T]$, some  bounded domain $\omega\subset \R^2$ (or $\omega\subset\R$) that has a  Lipschitz boundary and some $\delta\in (0,1)$. The time-dependent fluid domain is defined by
$$\Omega_\eta(t):=\{(x,y) \in {\omega} \times (0,\infty):0 \le y \le \eta(t,x)\},~t \in [0,T].$$ 
Here and in the following $x$ denotes a $2D$ (or $1D$), $y$ a $1D$ and $z=(x,y)$ a $3D$ (or $2D$) variable.
With some misuse of notation we consider the space-time domain 
$$[0,T] \times \Omega_\eta(t):=\bigcup_{t \in [0,T]}\{t\} \times \Omega_\eta(t).$$
The motion of the fluid is described by the incompressible Navier-Stokes equations
\begin{align}
\label{eq:1}
\rho_f (\partial_t v+[\nabla v] v) &=\mu_f\Delta v-\nabla p+\rho_f f& \text{ on }[0,T] \times \Omega_\eta(t),\\
\label{eq:2}
\di v&=0&\text{ on } [0,T] \times \Omega_\eta(t),
\end{align}
where the fluid's velocity field $v$ and the pressure $p$ are the unknown quantities, $\rho_f$ is the fluid density, $\mu_f$ the fluid viscosity and $f$ is a given outer force (e.g.\ gravity). By
 $\sigma(v,p)=2\mu_f \nablasym v -p\mathbb{I}$ we denote the fluid stress tensor, where $\nablasym v:=\frac12(\nabla v+(\nabla v)^T)$ is the symmetric part of the gradient and $\mathbb{I}$ denotes the identity matrix in 3D, (2D). The incompressibility condition implies that the pressure is determined by the velocity field. On the non-moving parts of the container $B_c=\omega \times \{0\}\cup \partial \omega\times [0,1]$ we assume no-slip boundary conditions 
\begin{equation}
\label{eq:5}
v=0 \text{ on }[0,T] \times B_c.
\end{equation}
The moving part of the shell satisfies a linearized plate equation of Koiter type with a source term stemming from the forces the fluid exerts on the shell 
\begin{equation}
\label{eq:33}
\rho_s h_0 \partial_t^2\eta+\Lcal(\eta,\partial_t\eta)=\mathcal{F}(u,p,\eta)+\rho_s g,~\text{on}~[0,T] \times {\omega},
\end{equation}
with Dirichlet boundary conditions
\begin{equation}
\label{eq:6}
\eta=1,~\nabla\eta=\Delta \eta=0~ \text{on}~(0,T) \times \partial{\omega}.
\end{equation}
Here $\eta$ is the (scalar valued) unknown deformation, $\rho_s$ is the solid density, $h_0$ is the thickness of the plate, $\Lcal(\eta,\partial_t\eta)$ is the $L^2$ gradient of the elastic and dissipative potentials of the deformation of the plate, $\mathcal{F}$ are forces stemming from the fluid and $g$ is a given outer force. Due to the troubles between hyperbolic equations and non-linearities we have to assume that $\Lcal(\eta,\partial \eta)$ is of the following form
\[
\Lcal(\eta,\partial_t\eta):= \alpha\Delta^2 \eta - \tilde{\beta} \Delta\eta -\tilde{\gamma} \Delta \partial_t \eta
\]
with $\alpha>0$ and $\tilde{\beta},\tilde{\gamma}\geq 0$.
Note that the equations for the fluid are stated in Eulerian coordinates while the equations for the solid are stated in Lagrangian coordinates.

The fluid and the shell are coupled via a kinematic and a dynamic coupling condition on the moving interface. For expressing the coupling condtions we define the variable transform from Langrangian to Eulerain coordinates
$$ \psi: [0,T] \times {\omega} \to [0,T] \times \R^3,\quad (t,x) \mapsto (t,x,\eta(t,x)).$$
The dynamic coupling condition states that the total force in normal direction at the interface is zero
\begin{equation}
\label{eq:9}
\mathcal{F}(v,\eta,p)=-(0,1)^t((\nabla v -p\mathbb{I})\circ\psi)n\cdot n~\text{on}~[0,T]\times {\omega},
\end{equation}
where $n(t,x)=(-\nabla \eta,1)/(1+|\nabla\eta|^2)^\frac{1}{2}$ is the outer normal of $\Omega_\eta(t)$ at the point $(x,\eta(x))$.

We assume a no slip kinematic boundary condition, i.e. the fluid and the structure velocity are equal at the interface
\begin{equation}
\label{eq:4}
v\circ \psi =(0,\partial_t \eta)^T~\text{on}~[0,T]  \times {\omega},
\end{equation}
To complete the equations we impose initial conditions
\begin{align}
\label{eq:7}
v(0)&=v_0 \text{ on } ~\Omega_\eta(0),\\
\label{eq:8}
\eta(0)=\eta_0,~~ \partial_t \eta(0)&=\eta^* \text{ on }~ {\omega}.
\end{align}
We will refer to \eqref{eq:1}-\eqref{eq:8} as \textit{FSI} in the following.

By formally multiplying equation \eqref{eq:1} by $v$, \eqref{eq:33} by $\partial_t \eta$ and integrating over $\Omega_\eta(t)$, $\omega$ and $(0,t)$ we get (using Korn's identity Lemma~\ref{lem:korn} and Absorption) the energy inequality
\begin{align} \begin{aligned}
\label{energy}
&\|v(t)\|_{L^2(\Omega_{\eta(t)})}^2+\|\partial_t \eta(t)\|^2_{L^2(\omega)}+\|\nabla^2 \eta(t)\|^2_{L^2(\omega)}
+ \int_{0}^t \|\nabla v(\tau)\|^2_{L^2(\Omega_{\eta(\tau)})} d\tau 
\\
&\quad \le c\bigg( \|v_0\|^2_{L^2(\Omega_{\eta_0})}+\|\eta_1\|^2_{L^2(\omega)}+\|\nabla^2 \eta_0\|^2_{L^2(\omega)}+
+ \int_0^t\norm{f(\tau)}^2_{L^2(\Omega_{\eta(\tau)})}+\norm{g(\tau)}_{L^2(\omega)}\, d\tau\bigg).
 \end{aligned} \end{align}
In the paper we use the standard notation for Lebesgue and Sobolev spaces. 
The weak solutions to \textit{FSI} are defined in the following function spaces.
\begin{align*}
\mathcal{V}_\eta(t)&=\{v \in H^1(\Omega_\eta(t)): \di v=0\text{ in }\Omega_{\eta(t)},v=0~\text{on}~B_c\},\\
\mathcal{V}_F&=L^\infty((0,T),L^2(\Omega_\eta(t))\cap L^2((0,T),\mathcal{V}_\eta(t)),\\
\mathcal{V}_K&=W^{1,\infty}([0,T],L^2({\omega}))\cap \{\eta \in L^\infty([0,T],H^{2}({\omega})):\eta=1,~\nabla\eta{=\Delta \eta}=0~\text{on}~\partial\omega\} 
\\
\mathcal{V}_S&=\{(v,\eta) \in \mathcal{V}_F\times \mathcal{V}_K:v\circ \psi=\partial_t \eta\},\\
\mathcal{V}_T&=\{(w,\xi) \in \mathcal{V}_F\times \mathcal{V}_K:w \circ \psi=\xi,~ \partial_t w \in L^2(0,T;L^2(\Omega_\eta(t)))\}
\end{align*}
For the distributional time derivative we introduce the following space
\[
\tilde{W}^{-l,p'}(\Omega):=(\{f\in W^{1,p}(\Omega)\, :\, f=0\text{ on }B_c\})^*.
\]
\begin{Definition} Let $f\in L^2([0,T]\times \omega\times \R)$, $g\in L^2([0,T]\times \omega$, $\eta_0\in H^2_0(\omega)$, $\eta^*\in L^2(\omega)$ and $v_0\in L^2(\Omega_{\eta_0})$. 
Then we call a pair $(v,\eta) \in \mathcal{V}_S$ a weak solution to \textit{FSI} if it satisfies the energy inequality \eqref{energy}, if
\begin{align}
\label{eq}
\begin{aligned}
&\frac{d}{dt}\left(\rho_f\int_{\Omega_\eta(t)} v \cdot w dz\right)-\rho_f\int_{\Omega_\eta(t)} v\cdot \partial_t w -2\mu\nablasym v : \nablasym w+\rho_f(v \otimes v) : \nabla w dz\\
&\quad +h_0 \rho_s \partial_t\left(\int_{{\omega}}\partial_t \eta \xi ~dx\right)-h_0 \rho_s\int_{{\omega}} \partial_t \eta \xi_t\,dx+  \skp{\Lcal(\eta,\partial_t\eta)}{\xi}= \rho_f\int_{\Omega_{\eta}}f\cdot w\, dz + \rho_s\int_{{\omega}}g\xi\, dx
\end{aligned}
\end{align}
for all $(w,\xi) \in \mathcal{V}_T$ as an equation in $\mathcal{D}'(0,T)$ and if it attains the initial conditions in the sense of the $L^2$ weak convergence.
\end{Definition} 
%
%

\subsection{Main results}
Our main result is the following.\footnote{For the notation please see the next section.} 
\begin{Theorem}
	\label{the:1}
	In case that $\omega\subset \R^2$ let $r>2$ and $s>3$ and in case that $\omega\subset \R$ let $r=2$ and $s=2$.  
Assume that $(v_2,\eta_2)$ is a weak solutions to \textit{FSI} on $[0,T]$, such that $\min_{[0,T]\times {\omega}}\eta_2>0$ and additionally that $v_2 \in L^{r}(0,T;W^{1,s}(\Omega_{\eta_2})) 
	$ and $\partial_t v_2 \in L^2(0,T;\tilde{W}^{-1,r}(\Omega_{\eta_2}))$. Then this solution is unique in the class of weak solutions.
In particular, if $(v_1,\eta_1)$ is any weak solution to \textit{FSI} on $[0,T_0]$ (for any $T_0>0$) and if $v_1(0)=v_2(0)$, $\eta_1(0)=\eta_2(0)$, $\partial_t\eta_1(0)=\partial_t \eta_2(0)$, than $(v_1,\eta_1)\equiv(v_2,\eta_2)$ as an equation in $\mathcal{V}_S$ on $[0,T_0]$.
\end{Theorem}
In some situations strong solutions are known to exist. In particular, in the case of $\omega=[0,L]$ and $\tilde{\gamma}>0$ strong solutions exist for arbitrary times~\cite{GraHil16}. This means that our result implies the following corollary.
\begin{cor}
In the 2D case ($\omega=[0,L]$) with $\tilde{\gamma}>0$ and smooth initial values, there exists a strong solution to \textit{FSI} which is unique in the class of weak solutions.
\end{cor}

\begin{rem}[Minimality of the regularity assumptions on $v_2$.]
\label{rem:serrin}
Let us compare our assumptions to the case of a non-moving domain for a 3D fluid; i.e.\ $\eta\equiv\eta_c$ and therefore $\Omega_{\eta_c}\subset \R^3$ is constant in time and $v_1$, $v_2 \in\mathcal{V}_F$ are weak (Leray-Hopf) solutions. If additionally $v_2$ satisfies the Ladyzhenskaya-Prodi-Serrin condition, namely $v_2 \in  L^r(0,T;L^{q}(\Omega))$ for $\frac{3}q+\frac{2}r=1$, 
then from the well known regularity and uniqueness result~\cite{Pro59,Ser62,Ser63,Lad67,Sve03} on the Navier-Stokes equations it follows:
$$\|w(t)\|^2 \le C\|w(0)\| \exp\left(c\int_0^t \|v_2\|_{L^{q}}^rdy\right)$$
which in particular, implies the weak-strong uniqueness. In order to obtain the above estimate a {\em regularity theory} for solutions satisfying the Ladyzhenskaya-Prodi-Serrin condition is used. 

For the here considered fluid-structure interactions a regularity theory for weak solutions satisfying the Ladyzhenskaya-Prodi-Serrin condition is not known to be satisfied up to date. Actually, it is debatable whether such a theory can expected to be true. (This counts even for 2D fluid-structure interactions in case when $\tilde{\gamma}=0$.) However, the borders for the exponents in our assumptions have the same index as in the exponents in the Ladyzhenskaya-Prodi-Serrin condition. We briefly explain this here: We assume in 3D that the stronger solution satisfies $v_2 \in L^r(0,T;W^{1,s}(\Omega))$ for some $s>3$ and $r>2$. As $W^{1,s}(\Omega) \hookrightarrow L^\infty(\Omega)$ for all $s>3$ the corresponding borderline exponent for $v_2$ is the one of $L^2(0,T; L^\infty(\Omega))$ which has the index $\frac{3}q+\frac{2}r=\frac{3}{\infty}+\frac{2}{2}=1$.

Please observe that in $2D$ no further assumption on the gradient are necessary beyond its energy estimate. Our stronger assumptions on the weak time-derivative are necessary both in 2D and 3D. This is again due to the fact that a regularity theory for solutions satisfying the Ladyzhenskaya-Prodi-Serrin condition might not be valid.  While the bounds on the {\em index} for the spaces we request for the weak time-derivative (of the stronger solution) are in coherence with weak solution, we have to assume that the negative space is considerably smaller; i.e.\ the dual of the Sobolev space and not the dual of its solenoidal subspace. 
\end{rem}

Further we prove the following stability estimate.
\begin{Theorem}
	\label{the:2}
Let $(v_2,\eta_2)$ be weak solutions to \textit{FSI} on $[0,T]$, such that $\min_{[0,T]\times {\omega}}\eta_2>0$ and that additionally $v_2 \in L^{r}(0,T;W^{1,s}(\Omega_{\eta_2}))$ 
	 and $\partial_t v_2 \in L^2(0,T;\tilde{W}^{-1,r}(\Omega_{\eta_2}))$  for any  $s>3$ and any $r>2$. If $(v_1,\eta_1)$ is a weak solution to \textit{FSI} on $[0,T]$, then for $\tilde{v_2}(t,x,y)=v_2(t,x, y \frac{\eta_1(t,x)}{\eta_2(t,x)})$ we find that
	\begin{align*}
&\sup_{t\in [0,T]}\|(v_1-\tilde{v}_2)(t)\|_{L^2(\Omega_{\eta_1(t)})}^2+\|\partial_t(\eta_1-\eta_2)(t)\|_{L^2(\omega)}^2+\|
(\eta_1-\eta_2)(t)\|^2_{H^2(\omega)}
\\
&\qquad 
+ \int_{0}^T \|(v_1-\tilde{v}_2)(\tau)\|_{H^1(\Omega_{\eta_1(\tau)})}^2 d\tau 
\\
&\quad \le C (\|v_1^0-\tilde{v}_2^0\|_{L^2(\Omega_{\eta_1^0})}^2
+\|\eta_1^*-\eta_2^*\|_{L^2(\omega)}^2)
+\|(\eta_1^0-\eta_2^0)\|^2_{H^2(\omega)}
\\
&\quad  
+C \int_0^T \|(f_1-\tilde{f}_2)(\tau)\|_{H^1(\Omega_{\eta_1(\tau)})}^2+\|(g_1-g_2)(\tau)\|_{L^2(\omega)}^2 d\tau ,
	\end{align*}
	where the constant depends on $\omega,T$, the assumed bounds on $v_2$, the $L^2$-bounds of $f_1,f_2$ and (symmetrically) on the two deformations $\eta_1,\eta_2$ via the bounds related to the energy estimates and via Theorem~\ref{thm:boris}. 
\end{Theorem}
In particular, the constant $C$ can be bounded a-priori in dependence of $\omega,T$, the assumed bounds on $v_2$ and the right hand side of the energy inequality~\eqref{energy} for both solutions.

\subsection{Analytical strategy \& technical novelties}
Usually for uniqueness (or stability estimates) one takes the difference of the two solutions or, in case of a hyperbolic evolution, its time-derivative as a test function. We wish to emphasize that due to the variable geometry {\em depending on the solution}, even uniqueness of strong solutions for longer times (provided they exist) does not follow in a straight forward manner. An additional difficulty regarding weak-strong uniqueness results is that the regularity of one solution is too low to be used as a test function. We follow the approaches developed in~\cite{Sta05,brenier2011weak,CheMuhNec17}.
 The idea is to resolve the difference of the systems tested by the difference of solutions into the energy inequality of the weak solution and terms containing a coupling where at least one function is sufficiently regular.

In order to make one fluid velocity a test function for the other equation we follow the methodology introduced in~\cite{Gui10} where a change of variables from one geometry to the other is introduced that conserves the solenoidality property. This {suffices} to circumvent the weak regularity properties of the pressure in case of incompressible fluids.\footnote{In unsteady incompressible problems the pressure is known to be hard to control w.r.t.\ the time variable even in the simplest case of Stokes equation in a fixed (smooth) geometry~\cite{KocSol02}.} 
What can not be circumvented is the {\em weak regularity of the time-derivative of the involved test-functions}. The {\em technical highlight} is a mollification-in-time operator that conserves solenoidality in the variable domains and the coupling of the boundary conditions. Moreover, it does not reduce the regularity (in space) significantly. The operator is introduced in Lemma~\ref{lem:molly}. A result that might be of independent interest is that this mollification can be used to show that all weak solutions do indeed have a distributional time derivative in a Bochner space involving negative Sobolev spaces (see Proposition~\ref{pro:time}). Finally, of further use in the future might be the estimates (especially on the convective term) which were necessary in order to stay with our assumptions that close to the Ladyzhenskaya-Prodi-Serrin conditions.
\subsection*{Acknowledgments}
\noindent
S.~Schwarzacher and M.~Sroczinski thank the support of the primus research programme PRIMUS/19/SCI/01 and the University Centre UNCE/SCI/023 of Charles University. Moreover they thank for the support of the program GJ19-11707Y of the Czech national grant agency (GA\v{C}R).


\section{Notation \& preliminary results}

\subsection{Simplifications}
In order to simplify the quite technical argument below we assume in the following that $\Lcal(\eta,\partial_t\eta)\equiv \Delta^2\eta$; as the argument can be adapted to more general $\Lcal$ in a straight forward manner. Moreover we will assume in the following that we have a fluid in $3D$. In particular we assume that $\omega\subset \R^2$. The adaption of the proof for $\omega\subset \R$ implies only simplifications and no further complications. Finally we set all constants in the equations to one (i.e.\ both densities, the thickness of the plate, the viscosity of the fluid).

 For vector valued functions $u:\Omega_\eta\to \R^3$ we use $u=(u',u^3)^T=(u^1,u^2,u^3)^T$.
The constants $c, c_1,...$ are used as a constants that are independent of $\eta$, while the constants $C, C_1,...$ are used as constants that may depend on bounded quantities of the deformations. Both letters $c,C$ may change there actual value with every instance. Moreover, we use the notation $a\sim b$, if there are constants $c, c_1$ 
such that $\abs{a}\leq c\abs{b}\leq c_1\abs{a}$.

\subsection{Identities \& Estimates}
We will use Reynold's transport theorem which for plates reads (using the fact that the third component of the outer normal times the Jacobian of the change of variables is one) as for all $u\in W^{1,1}(0,T;\Omega_{\eta})$ with $u'(x,\eta(x))=0$ for all $x$, we find 
\[
\partial_t \bigg(\int_{\Omega_\eta} u(t,z)\cdot \phi(t,z)\, dz\bigg)=\int_{\Omega_\eta} \partial_t(u\cdot \phi)\, dz +\int_\omega u^3(t,x,\eta(x)) \phi^3(t,x,\eta(x)))\partial_t\eta(t,x)\, dx,
\]
for all $\phi,\eta$ for which the above expression is well defined. 

Next due to the zero boundary conditions of $v'$ on $\partial\Omega$ we actually may use Korn's identity which is done throughout the paper.
\begin{lem}
\label{lem:korn}
Let $u\in H^1(\Omega_\eta)$ such that $u=0$ on $B_c$ and $u'(x,\eta(x))=0$, than
\[
\norm{u}_{H^1(\Omega_\eta)}\sim \norm{\nabla u}_{L^2(\Omega_\eta)} = 2\norm{\nablasym u}_{L^2(\Omega_\eta)}.
\]
\end{lem} 
\begin{proof}
The fact that $\norm{u}_{H^1(\Omega_\eta)}\sim \norm{\nabla u}_{L^2(\Omega_\eta)}$ follows by Poincar\'{e}'s inequality as all components have zero boundary values on large parts of the boundary and the inequality is a straight consequence of the fundamental theorem of calculus. Korn's identity follows by~\cite[Lemma 4.1]{MuhSch19}.
\end{proof}
Our proof makes use of the following additional regularity result that has been shown in~\cite[Theorem~1.2]{MuhSch19}:
\begin{Theorem}
\label{thm:boris}
For any weak solution to \textit{FSI} we find that as long as $\eta>0$ in $[0,T]\times {\omega}$ that $\eta\in L^2(0,T;H^{2+\sigma}({\omega}))$ and $\partial_t\eta\in L^2(0,T;H^{\sigma}({\omega}))$ for all $\sigma<\frac{1}{2}$.
\end{Theorem}
An adaption of \cite[Theorem~1.2]{MuhSch19} is the following corollary. 

We will need the following interpolation estimate:
\begin{enumerate}
\item $L^{a'}(L^a)\subset L^\infty(L^1)\cap L^2(L^2)$ for all $a\in [1,2]$.
\end{enumerate}
\begin{lem}
\label{lem:interpol}
 For $Y\subset\mathbb{R}^2$. If $b\in L^\infty(0,T;L^2(Y))$ and $\phi\in L^2(0,T;W^{1,a}(Y))$ for all $a\in (1,2)$, then $\abs{b}\abs{\phi}\in L^2(0,T;L^p(Y))$ for all $p\in (1,2)$.
 \end{lem}
\begin{proof}
The result follows by Sobolev embedding and H\"older's inequality.
\end{proof}

Very often we will have the product of a function defined on ${\omega}$ with a function defined on $\Omega_{\eta}$. We will integrate such products over $\Omega_\eta$ where one of the two functions is than constant in the variable direction. In some cases this allows to improve the regularity. In particular we will need the following extra information on the weak solution that will be used upon the convective term:
\begin{lem}
\label{lem:convective}
Let $(\eta,v)$ be a weak solution to FSI. 
Then we find that
$
\int_{0}^{\eta (t,x)}\abs{v }\, dy\in L^2(0,T;H^1({\omega}))
$
\[
\normlr{\int_{0}^{\eta (t,x)}\abs{v }\, dy}_{L^2(0,T;H^1({\omega}))}\leq c\norm{v}_{L^2(0,T;H^1(\Omega_\eta))}\norm{\eta}_{L^\infty(0,T;H^2(\omega))}
\]
and
$
\int_{0}^{\eta (t,x)}\abs{v }^2\, dy\in L^2(0,T;W^{1,1}({\omega}))
$
\begin{align*}
\normlr{\int_{0}^{\eta (t,x)}\abs{v(t) }^2\, dy}_{L^2(0,T;W^{1,1}({\omega}))}&\leq
\|v\|_{L^2(0,T;L^2(\Omega_\eta))}+2\norm{v}_{L^\infty(0,T;L^2(\Omega_\eta))}\norm{\nabla v}_{L^2([0,T]\times \omega)}\\
& \quad+\|\partial_t \eta\|_{L^\infty(0,T;L^2(\omega))}^2\|\nabla \eta\|_{L^1(0,T;L^\infty(\omega))}.
\end{align*}
This implies in particular that $\int_{0}^{\eta (t,x)}\abs{v }^2\, dy\in L^2([0,T]\times {\omega})$.
\end{lem}
\begin{proof}
For the first statement we calculate
\[
	\nabla_x \int_{0}^{\eta (t,x)}\abs{v(x,y)}\, dy= \int_{0}^{\eta (t,x)}\nabla_x \abs{v(x,y)}\, dy+\nabla_x \eta(t,x)\abs{\partial_t \eta}
	\]
which is uniformly bounded in $L^2([0,T]\times \omega)$ since $v_1\in L^2(0,T; H^1(\Omega_\eta)$, $\partial_t\eta\in L^2(0,T;L^3(\omega))$ and $\nabla \eta\in L^\infty(0,T;L^6(\omega))$. The estimate follows using Sobolev embedding and the trace theorem~\cite[Lemma~6]{BreSch18}.

\seb{For the second statement we calculate
 \begin{align*}
 \nabla_x\int_{0}^{\eta (t,x)}\abs{v }^2\, dy&=
\int_{0}^{\eta (t,x)}2[\nabla v ]v \, dy+\abs{v (\eta(t,x))}^2\nabla \eta (t,x)
\\
&=
\int_{0}^{\eta (t,x)}2[\nabla v ]v \, dy+\abs{\partial_t\eta (t,x)}^2\nabla \eta (t,x)=:I_1+I_2
 \end{align*}
Due to Holeder's inequality 
$$\int_\omega I_1 \le 2\|v\|_{L^2(\Omega_\eta)}\|\nabla v\|_{L^2(\Omega_\eta)}.$$
And it is also straightforward to see
$$\int_\omega I_2 \le \|\partial_t \eta\|_{L^2(\omega)}^2\|\nabla \eta\|_{L^\infty(\omega)}.$$
Thus the statement follows 
since $v \in L^\infty(0,T;L^2(\Omega_{\eta }))\cap L^2(0,T;H^{1}(\Omega_{\eta }))$, $\partial_t \eta \in L^\infty(0,T,L^2(\omega))$ and by Theorem \ref{thm:boris} $\eta \in L^2(0,T;H^{2+\sigma}(\omega)) \hookrightarrow L^2(0,T;W^{1,\infty}(\omega))$ for all $\sigma >0$.}
\end{proof}
%

\subsection{Convolution}
Since the regularity in space of $\partial_t\eta$ and the regularity in time for $v$ a a test function is formally not sufficient to use the couple as a test function we have to introduce a mollification in time. Unfortunately, it was not possible to use the mollification introduced~\cite{MuhSch19} and we have to introduce a new version. Already here the \textit{regularity of the deformation influences the regularity of the mollification sensitively} due to the fact that a change of variables will be a part of the convolution kernel.

First a technical Lemma. Here we will use a mollifier with respect to time. As is the standard procedure, choose a function $j \in C_0^\infty(\R)$ which is positive, even, has support in $(-1,1)$ and satifies $\int_\R j ~dt=1$, $\frac{d}{dt}j(-t) \ge 0$, $\frac{d}{dt}j(t) \le 0$ for $t \ge 0$. For $\delta>0$ define $j_\delta(t) \equiv \delta^{-1}j(t/\delta)$. Then $j_\delta$ has support in $(-\delta, \delta)$ and otherwise the same properties as $j$.

Let $(H,(\cdot,\cdot))$ be a Hilbert space, $T>0$. Let $u \in L^\infty(0,T;H)$ be continues w.r.t. the weak topology on $H$ and assume that the limits $u(0):=\lim_{t\to 0} u(t)$, $u(T):=\lim_{t \to T} u(t)$ exist in the weak topology of $H$. In the following we will call the space of all such functions $C_{w}(0,T;H)$. Define the extension $\bar{u} \in L^\infty(\R,H)$ by
\begin{align}
\label{extend}
\bar{u}^T(t)=\begin{cases}
u(t), & t \in (0,T),\\
u(0), & t \in (-\infty,0],\\
u(T), & t \in [T,\infty].
\end{cases}
\end{align}
Now for all $\delta>0$, $t \in [0,T]$ set
$$u_\delta^T(t)=\int_\R j_\delta(\tau-s)\bar{u}^T(s)ds.$$
It is well known that $u_\delta^T \in C^\infty([0,T],H)$ and $\lim_{\delta \to 0} u_\delta=u$ in $L^p(0,T;H)$ for all $1 \le p < \infty$. Furthermore the following holds

\begin{lem}
	\label{lemma}
	Let $u,v \in \seb{C_w}(0,T;H)$ and $t \in (0,T]$. Then for all $t \in [0,T]$
	\begin{equation}
	\label{eq1}
	\lim_{\delta\to 0}\int_0^t(u,v^T_\delta)-(u_\delta^T,v)d\tau=0
	\end{equation}
	and 
	$$
	\lim_{\delta \to 0} \int_0^T\left(u,\frac{d}{dt}v_\delta^T\right)+\left(\frac{d}{dt}u_\delta^T,v\right)d\tau=(u(T),v(T))-(u(0),v(0))
	$$
\end{lem}

\begin{proof}
In the following we omit the superscript $T$.
The first assertion holds since
	$$
	(u,v_\delta)-(u_\delta,v)=(u,v_\delta-v)+(v,u-u_\delta).$$
	and the weak continuity in time.
	
	To prove the second assertion note that $\seb{\partial_t}j_\delta$ is an odd function and therefore
	\begin{equation*}
	\int_0^T\int_0^T \frac{d}{d\tau}j_\delta(\tau-s)(v(s),u(\tau))~dsd\tau=-\int_0^T\int_0^T \frac{d}{d\tau}j_{\delta}(\tau-s)(u(s),v(\tau))~d\tau ds.
	\end{equation*}
	Hence
	\begin{align*}
	&\int_0^T\left(u,\frac{d}{dt}v_\delta\right)+\left(\frac{d}{dt}u_\delta,v\right)d\tau\\
	&\quad =\int_0^T\left(u(\tau),\int_{-\infty}^0 \frac{d}{d\tau}j_\delta(\tau-s) \bar{v}(s)ds\right)~d\tau+\int_0^T \left(u(\tau),\int_{T}^\infty \frac{d}{d\tau}j_\delta(\tau-s) \bar{v}(s)ds\right)d\tau\\
	&\qquad + \int_0^T\left(v(\tau),\int_{-\infty}^0 \frac{d}{d\tau}j_\delta(\tau-s) \bar{u}(s)ds\right)~d\tau+\int_0^T \left(v(\tau),\int_{T}^\infty \frac{d}{d\tau}j_\delta(\tau-s) \bar{u}(s)ds\right)d\tau\\
	&\quad:=R_1(\delta)+R_2(\delta)+R_3(\delta)+R_4(\delta).
	\end{align*}
By symmetry it suffices to prove $R_1(\delta)\to -\frac{1}{2}(u(0),v(0))$ and $R_2(\delta) \to \frac{1}{2}u(t)v(t)$.
	As $\bar{v}(s)\equiv v(0)$ for all $s<0$ and $j_\delta$ has support in $(-\delta,\delta)$ we get
	 \begin{align*}
	R_1(\delta)=\int_0^T(v(0),u(\tau))\int_\tau^\infty \frac{d}{ds}j_\delta(s)~dsd\tau
	=\int_0^\delta (v(0),u(\tau))\int_{\tau}^\delta \frac{d}{ds}j_\delta(s)~dsd\tau\\
	=\int_0^\delta (v(0),u(\tau)) (j_\delta(\delta)-j_\delta(\tau))~d\tau
	=-\frac{1}{\delta}\int_0^\delta (v(0),u(\tau))j\left(\frac{\tau}{\delta}\right)~d\tau\\
	=-\int_0^1(v(0), u(\delta\tau))j(\tau)d\tau.
	 \end{align*}
	By weak continuity we get
	$$\lim_{\delta \to 0}(v(0),u(\delta\tau))j(\tau) = (v(0),u(0))j(\tau).$$
	As $u \in L^\infty(0,T;H)$ we get by dominated convergence 
	$$\lim_{\delta \to 0} R_1(\delta)= -(v(0),u(0))\int_0^1j(\tau)~d\tau=-\frac{1}{2}(v(0),u(0)).$$
	The convergence of $R_2(\delta)$ is analogous.
\end{proof}
\seb{Here and in the following we will always consider the extension $\overline{u}$ introduced above implicitly. Meaning, that when ever necessary we extend any function to a global in (positive and negative) time object.}  In order to treat distributional time derivatives we will use the notation of the dual product over a variable domain by
\[
\int_0^T\skp{f}{\phi}_{\eta}\, dt:=\int_0^T\skp{f(t)}{\phi(t)}_{\Omega_{\eta(t)}}\, dt,
\]
where $\skp{f}{\phi}_{\Omega_{\eta(t)}}$ is the dual product over function spaces over $\Omega_{\eta(t)}$ which are assumed to be bilinear mappings that map into measurable functions in time. 

For our case of moving boundaries we will need the following convolution result that allows to con volute with respect to the moving geometry by keeping the solenoidality.

\begin{lem}
\label{lem:molly}
	Let $\eta \in \mathcal{V}_K$, such that $\eta$ is bounded uniformly from below.	
	Let $\phi\in L^{\seb{\kappa}}(0,T;L^q(\Omega_\eta))\cap L^\alpha(0,T;W^{1,a}(\Omega_\eta))$ for some $a>1$ and $\alpha,\seb{\kappa},q\geq 1$.
	Let $b \in L^2(0,T;L^1({\omega}))$ with $\phi(t,x,\eta(x))=(0,b(t,x))$ on $[0,T] \times \omega$ (in the sense of traces).
	
	Set $K:[0,T]\times [0,T]\times \R\times {\omega}\to \R^{3\times 3}$
	$$
	K(s,t,y,x)=\begin{pmatrix}
	\frac{\eta(s,x)}{\eta(t,x)} & 0 &0
	\\
	0 & \frac{\eta(s,x)}{\eta(t,x)} & 0
	\\
	-y\partial_{x_1}\big(\frac{\eta(s,x)}{\eta(t,x)}\big)& -y\partial_{x_2}\big(\frac{\eta(s,x)}{\eta(t,x)}\big)& 1
	\end{pmatrix}$$ 
	For each $\delta>0$ define $b_\delta=b*j_\delta$ and 
	$$ 
	{\phi}_\delta(t,x,y)=\int_0^T K(s,t,y,x)\phi\left(s,x,y\frac{\eta(s,x)}{\eta(t,x)}\right)\seb{j_\delta}(t-s)~ds.
	$$
	Then it holds for $\seb{\kappa}<\infty$ that 
	$$
	\di {\phi}_\delta=0,\quad {\phi}_\delta(t,x,\eta(x))=b_\delta(t,x)
	$$
	and ${\phi}_\delta \to \phi$ strongly $L^{\seb{\kappa}}(0,T;L^p(\Omega_\eta(t)))$ for all $p\in [1,\seb{\kappa})$.

	Moreover, 
\begin{enumerate}
\item if $\phi\in  L^2(0,T;W^{1,a}(\Omega_\eta))$  for all $a\in(1,2)$, then ${\phi}_\delta \to \phi$ converges weakly in  $L^2(0,T;W^{1,p}(\Omega_\eta(t))$ for all $p\in [1,2)$.
 \item if $\phi\in  L^2(0,T;W^{1,a}(\Omega_\eta))$ for $a>3$ than ${\phi}_\delta \to \phi$ converges weakly in  $L^2(0,T;H^1(\Omega_\eta(t))$.
\item if $\phi\in H^1(0,T;\tilde{W}^{-1,p'}(\Omega_\eta))\cap L^2(0,T;W^{1,a}(\Omega_\eta))$ for some $a>3$ and some $p\in (1,2)$ then  $\partial_t{\phi}_\delta\to \partial_t\phi$ converges weakly in 
$L^2(0,T;\tilde{W}^{-1,p'}(\Omega_\eta)))$.
\end{enumerate}	

\end{lem}

\begin{proof}
	We define 
	$${\phi}(s,t,x,y)=K(s,t,x,y)\phi\left(s,x,y\frac{\eta(s,x)}{\eta(t,x)}\right)$$
	If we show that $\di {\phi} (t,s,x,y)\equiv 0$ then clearly also $\di {\phi}_\delta=0$. We get
	\begin{align*}
	\di{\phi}&=
	\left(\frac{\eta(s,x)}{\eta(t,x)}\right)\phi^1+\frac{\eta(s,x)}{\eta(t,x)} \di_x \phi^1+y\partial_y \phi^1  \nabla\big(\frac{\eta(s,x)}{\eta(t,x)}\big)\frac{\eta(s,x)}{\eta(t,x)}
	-\nabla\left(\frac{\eta(s,x)}{\eta(t,x)}\right) \phi^1
	\\
	&\quad -y\nabla\left(\frac{\eta(s,x)}{\eta(t,x)}\right)\frac{\eta(s,x)}{\eta(t,x)}\partial_y \phi^1+\frac{\eta(s,x)}{\eta(t,x)}\partial_y \phi^2
	=\frac{\eta(s,x)}{\eta(t,x)}(\partial_y \phi^2+\di_x \phi^1)=0,
	\end{align*}
	where we used in the last line that $\di \phi=0$.
	Now as $\phi(t,x,\eta(t,x))=(0,b(t,x))$ we get
	$${\phi}(s,t,x,\eta(t,x))=\phi(s,x,\eta(s,x))=(0,b(s)).$$
	Thus
	$${\phi}_\delta(t,x,\eta(t,x))=\int_0^T b(s)j_\delta(t-s)ds=b_\delta(t,x).$$
	\seb{For the convergence result we introduce the function on the reference domain
	\[
	\phi_0:[0,T]\times \omega\times [0,1]\to \mathbb{R}^3,\quad (t,x,y)\mapsto \phi(t,x,y\eta(t,x)).
	\]
	}
	 \seb{Let $ p\in [1,\kappa)$}. First we estimate ${\phi}_\delta^1-\phi^1$ in $L^\seb{\kappa}(0,T;L^p(\Omega_\eta(t)))$. We have
	\begin{equation*}
	({\phi}^1_\delta-\phi^1)(t,x,y)=\int_0^T (\frac{\eta(s,x)}{\eta(t,x)}\phi^1(s,x,y\frac{\eta(s,x)}{\eta(t,x)})-\phi^1(t,x,y))j_\delta(t-s)ds\\
	\end{equation*}
	Hence (by a change of variables) we find
	\begin{align*}
	&\int_0^T \left(\int_{\Omega_\eta(t)}|({\phi}^1_\delta-\phi^1)(t,x,y)|^pdxdy\right)^{\frac{\seb{\kappa}}{p}}dt
	\\
	&\quad =\int_0^T\left(\int_{\omega\times [0,1]}\left|\int_0^T(\eta(s,x)\phi^1(s,x,y\eta(s,x))-\eta(t,x)\phi^1(t,x,y\eta(t,x)))
	j_\delta(t-s)ds\right|^pdz\right)^{\frac{\seb{\kappa}}{p}}dt
	\\
	&\quad = \|\varphi_\delta-\varphi\|_{L^\seb{\kappa}(0,T;L^p(\omega \times [0,1]))}		
	\end{align*}
	\seb{for $\varphi(t,x,y)=\eta(t,x)\phi^1_0(t,x,y)$.} As $\eta \in L^\infty(0,T;L^\infty(\omega))$ and $\phi \in L^\seb{\kappa}(0,T;L^q(\Omega_\eta))$ this converges to $0$ by standard convolution estimates.
	Next note by a similar argument that
	\begin{align*}
	&\int_0^T\left(\int_{\Omega_\eta(t)}\left|\int_0^T(\phi^2\left(s,x,y\frac{\eta(s,x)}{\eta(t,x)}\right)-\phi^2(t,x,y))j_{\seb{\delta}}(t-s)~ds\right|^pdz\right)^{\frac{\seb{\kappa}}{p}}dt
	\\
	&\quad \le\|\eta\|_{L^\infty_t(0,T;L^\infty(\omega))}\|\phi^2_{0,\delta}-\phi^2_{0}\|_{L^\seb{\kappa}(0,T;L^p(\omega\times [0,1]))},
	\end{align*}
	 which also converges to $0$.
	Lastly 
	\begin{align*}
	&\int_0^T\left(\int_{\Omega_\eta(t)}\left|\int_0^T y\nabla\left(\frac{\eta(s)}{\eta(t)}\right)\phi^1\left(s,x,y\frac{\eta(s)}{\eta(t)}\right)j_{\seb{\delta}}(t-s)ds\right|^pdz\right)^{\frac{\seb{\kappa}}{p}}dt
	\\
	&=\int_0^T\left(\int_{\omega\times [0,1]}\left|\int_0^T y\seb{\eta(t)}\nabla\left(\frac{\eta(s)}{\eta(t)}\right)\phi^1_0(s,x,y)j_h(t-s)ds\right|^pdz\right)^{\frac{s}{p}}dt
	\end{align*}
	As $j_\delta$ has unit integral we can compute
	\begin{align*}
	&\int_0^T \seb{\eta(t)}\nabla\left(\frac{\eta(s)}{\eta(t)}\right)\phi^1_0(s)j_{\seb{\delta}}(t-s)ds
	=\int_0^T \phi_0^1(s)j_\delta(t-s)(\nabla \eta(s)-\nabla\eta(t))+\seb{\frac{\nabla\eta(t)}{\eta(t)}}(\eta(t)-\eta(s)))~ds
	\\
	&\quad =\int_0^T j_\delta(t-s)(\phi_0(s)\nabla\eta(s)-\phi_0(t)\nabla\eta(t))
	 +j_\delta(t-s)\seb{\frac{\nabla\eta(t)}{\eta(t)}}(\phi_0(s)\eta(s)-\phi_0(t)\eta(t))\\
	\\
	&\qquad +2j_\delta(t-s)\nabla\eta(t)(\phi_0(t)-\phi_0(s))~ds
	\end{align*}
	Thus
	\begin{align*}
	&\int_0^T\left(\int_{\Omega_\eta(t)}\left|\int_0^T y\nabla\left(\frac{\eta(s)}{\eta(t)}\right)\phi^1\left(s,x,y\frac{\eta(s)}{\eta(t)}\right)j_\delta(t-s)ds\right|^pdz\right)^{\frac{\seb{\kappa}}{p}}dt
	\\
	&\quad \le \|(\nabla\eta\phi^1_0)_\delta-\nabla\eta\phi^1_0\|_{L^2(0,T;L^p(\Omega_\eta))} +
	\int_0^T\left(\int_{\omega\times [0,1]}\seb{\frac{\abs{\nabla\eta(t)}^p}{\abs{\eta(t)}^p}} |(\eta\phi^1_0)_\delta(t)-\eta(t)\phi^1_0(t)|^pdz\right)^{\frac{\seb{\kappa}}{p}}dt
	\\
&\quad 	+\int_0^T\left(\int_{\omega\times [0,1]}|\nabla\eta(t)|^p|\phi^1_{0,\delta}(t)-\phi^1_0(t)|^pdz\right)^{\frac{\seb{\kappa}}{p}}dt
	\end{align*}
	The first term converges to $0$ by standard convolution. The third term we can estimate as $p<q$
	\begin{equation*}\int_0^T\left(\int_{\omega\times [0,1]}|\nabla\eta|^p|\phi_0^1*j_\delta-\phi_0^1|^p~dz\right)^{\frac{v}{p}}dt
	\le \|\nabla\eta\|_{L^\infty(0,T;L^{q*}(\omega))}\|\phi_{0,\delta}^1-\phi_0^1\|_{L^\kappa(0,T;L^{q}(\omega \times [0,1]))}.
	\end{equation*}
	 Hence this term converges to $0$ as well. The third term can be estimated analogously using the assumed uniform lower bounds on $\eta$.

	As we have shown strong convergence in $L^2(0,T;L^p(\Omega_\eta(t)))$ it suffices to show that $\nabla{\phi}_\delta$ is bounded in $L^2(0,T; L^{p}(\Omega_\eta(t)))$ to prove weak convergence.
	The estimate on the gradient is a standard exercise combining the bounds of $\eta$ and $\phi$ via H\"older's inequality. For that reason we omit here most of the details and only mention the critical terms that appear in the estimates.
	One critical term appearing in the estimates for  (1), (2), (3) can be estimated using
 \[
 	\abs{\nabla \phi}\abs{\nabla \eta}\in L^2(0,T;L^{p}(\Omega_\eta)\text{ for all }p\in [1,a).
 \]
	Moreover, one needs
	\begin{enumerate}
	\item[for (1)] $	\abs{ \phi}\abs{\nabla^2 \eta}\in L^2([0,T];L^{p} (\Omega_\eta))\text{ for all }p\in [1,2)$ by Lemma~\ref{lem:interpol}.
	\item[for (2)] $	\abs{ \phi}\abs{\nabla^2 \eta}\in L^2([0,T];L^{2} (\Omega_\eta))$ as $\phi\in L^2(L^\infty)$ by Sobolev embedding.
	\end{enumerate}

Next let us consider the weak time derivative. 
Let us take $\psi\in\tilde{W}^{1,p'}([0,T] \times{\omega}\times \R))$, such that $\psi(t,x,y)=0$ for all $x\in B_c$ and $\norm{\psi}_{W^{1,p'}([0,T]\times{\omega}\times \R)}\leq 1$ to find that 
	\begin{align*}
	\int_0^T\skp{\partial_t {\phi}_\delta}{\psi}
&= \int_0^T\int_0^T\skp{\partial_tK(s,t,y,x)\phi\Big(s,x,y\frac{\eta(s)}{\eta(t)}\Big)j_\delta(t-s)}{ \psi(t,z)}\,ds\, dt
\\
&\quad +\int_0^T\int_0^T\int_{\Omega_\eta}K(s,t,y,x)\phi\Big(s,x,y\frac{\eta(s)}{\eta(t)}\Big)\partial_tj_\delta(t-s)\cdot \psi(t,z)\, dz\,ds\, dt
\\
&\quad -\int_0^T\int_0^T\int_{\Omega_\eta}K(s,t,y,x)\partial_y\phi\Big(s,x,y\frac{\eta(s)}{\eta(t)}\Big)y\frac{\eta(s)}{\eta^2(t)}\partial_t\eta(t) j_\delta(t-s)\cdot \psi(t,z)\, dz\,ds\, dt
\\
&=(I)+(II)+(III)
	\end{align*}
The expression $(I)$ can be transferred into an integral by using partial integration in $x_i$ and the fact that $\phi^i(t,x,\eta(t,x))= 0$ for $i\in \{1,2\}$ and $(t,x)\in [0,T]\times \omega$:
\begin{align*}
(I)&=\sum_{i=1}^2\int_0^T\!\!\!\int_0^T\!\!\!\bigg(-\skp{y \partial_t\partial_{x_i}\Big(\frac{\eta(s,x)}{\eta(t,x)}\Big)\phi^i\Big(s,x,y\frac{\eta(s)}{\eta(t)}\Big)}{ \psi^3(t)}\,ds\, dt
\\
&\quad +\int_{\Omega_{\eta(t)}} \partial_t \Big(\frac{\eta(s,x)}{\eta(t,x)}\Big)\phi^i\Big(s,x,y\frac{\eta(s)}{\eta(t)}\Big)\cdot \psi^i(t,z)\, dz \bigg)j_\delta(t-s) \,ds\, dt
\\
&=\sum_{i=1}^2\int_0^T\!\!\!\int_0^T\!\!\!\int_{\omega} \partial_t \Big(\frac{\eta(s,x)}{\eta(t,x)}\Big) \bigg(\partial_{x_i}\int_0^{\eta(t,x)}y\phi^i\Big(s,x,y\frac{\eta(s)}{\eta(t)}\Big)\cdot \psi^3(t,x,y)\, dy
\\
&\quad + \int_0^{\eta(t,x)}y\phi^i\Big(s,x,y\frac{\eta(s)}{\eta(t)}\bigg)\cdot \psi^i(t,x,y)\, dy\bigg)\,dx j_\delta(t-s) \,ds\, dt 
\\
&=\sum_{i=1}^2\int_0^T\!\!\!\int_0^T\!\!\!\int_{\omega} \partial_t \Big(\frac{\eta(s,x)}{\eta(t,x)}\Big) \bigg(\int_0^{\eta(t,x)}y\partial_{x_i}\Big(\phi^i\Big(s,x,y\frac{\eta(s)}{\eta(t)}\Big)\cdot \psi^3(t,x,y)\Big)\, dy
\\
&\quad + \int_0^{\eta(t,x)}y\phi^i\Big(s,x,y\frac{\eta(s)}{\eta(t)}\bigg)\cdot \psi^i(t,x,y)\, dy\bigg)\,dx j_\delta(t-s) \,ds\, dt.
\end{align*}
But these expression can be estimated using that $p^*=\frac{3p}{3-p}$ can be assumed to be close enough to 6 such that
\begin{align*}
(I)&\leq C\int_0^T\norm{\partial_t\eta}_{L^2(\omega)}
\big((\norm{\phi}_{W^{1,s}(\Omega_{\eta(t)})}+\norm{\abs{\nabla\phi}\abs{\nabla \eta}}_{L^{3+(3-s)/2}(\Omega_{\eta})})\norm{\Psi}_{L^{p^*}(\Omega_{\eta})}+\norm{\Psi}_{W^{1,p}(\Omega_{\eta})}\big)\, dt.
\end{align*}
This expression is bounded as $\partial_t\eta\in L^\infty(0,T;L^2(\omega))$, $\abs{\nabla \eta}\abs{\nabla \phi}\in L^2(0,T;L^q(\Omega_\eta))$ for all $q\in [3,s)$.
The estimate on $(III)$ is analogous (but simpler).
 
 For $(II)$ we use $\partial_tj_\delta(t-s)=\partial_s j_\delta(t-s)$ to find (using the $0$-trace of $j_\delta(t-s)$ that)
	\begin{align*}
	(II)&=
	\int_0^T\int_0^T\partial_s\skp{(K(s,t,y,x)\phi\Big(s,x,y\frac{\eta(s)}{\eta(t)}\Big)j_\delta(t-s)}{ \psi(t,z)}\,ds\, dt
\\
&\quad -
\int_0^T\int_0^T\skp{\partial_sK(s,t,y,x)\phi\Big(s,x,y\frac{\eta(s)}{\eta(t)}\Big)j_\delta(t-s)}{\psi(t,z)}\,ds\, dt
\\
&\quad -\int_0^T\int_0^T\int_{\Omega_\eta}K(s,t,y,x)\partial_y\phi\Big(s,x,y\frac{\eta(s)}{\eta(t)}\Big)y\frac{\partial_s\eta(s)}{\eta(t)} j_\delta(t-s)\cdot \psi(t,z)\, dz\,ds\, dt
\\
&\quad -
\int_0^T\int_0^T\skp{K(s,t,y,x)\partial_s\phi\Big(s,x,y\frac{\eta(s)}{\eta(t)}\Big)j_\delta(t-s)}{\psi(t,z)}\,ds\, dt,
\\
&=: II_1+II_2+II_3+II_4.
	\end{align*}
	First observe, that$II_1=0$. The estimates on $II_2$, $II_3$ are similar to the estimate of $(I)$ above.  
	Now, finally $II_4$ is estimated using the assumption on $\partial_t\phi$. We define 
$\hat{K}^T(s,t,y,x)$ in such a way that
		\begin{align*}
	II_4&=-\int_0^T\int_0^T\skp{\partial_s\phi\Big(s,x,y\frac{\eta(s)}{\eta(t)}\Big)}{K^T(s,t,y,x)\psi(t,z)}_{\Omega_{\eta(t)}}j_\delta(t-s)\,ds\, dt
	\\
	&= -\int_0^T\int_0^T\skp{\partial_s\phi(s,z)}{\hat{K}^T(t,s,y,x)\psi\Big(s,x,y\frac{\eta(t)}{\eta(s)}\Big)}_{\Omega_{\eta(s)}}j_\delta(t-s)\,ds\, dt.
	\end{align*}
	This implies that
	\begin{align*}
	II_4\leq \int_0^T\int_0^T\norm{\partial_t\phi(s)}_{\tilde{W}^{-1,p'}(\Omega_{\eta})}
	\normlr{\hat{K}^T(t,s,y,x)\psi\Big(s,x,y\frac{\eta(t)}{\eta(s)}\Big)}_{W^{1,p}(\Omega_{\eta})}j_\delta(t-s)\,ds\, dt,
	\end{align*}
	which is uniformly bounded using $\abs{\nabla \psi}\abs{\nabla \eta}^2\in L^2(0,T;L^{p}(\Omega_\eta))$ and $ \abs{ \psi}\abs{\nabla^2 \eta}\in L^2([0,T];L^{p} (\Omega_\eta))$ for all $p\in [1,2)$.
	
\end{proof}

\subsection{The distributional time derivatives.}

 En passant we include here a result that is independent of our main result but might be important for further use. Here a meaning is given to the distributional time derivative of solutions. 

\begin{prop}
\label{pro:time}
Let $(v,p,\eta)$ be a weak solution satisfying \eqref{eq}, then if $v\in L^2(0,T;W^{1,s}(\Omega_{\eta}))$ for $s\geq 2$, than  
\[
\partial_t v+[\nabla v] v\in L^2(0,T;(W^{1,q}_{0,\di}(\Omega_{\eta})^*),
\]
for any $q\in (2,\infty)$ if $s=2$ and $q=2$ if $s>2$.

This means\footnote{The expression \eqref{eq:weak-time} seems to be the appropriate definition of a weak time derivative in the setting of fluid-structure interaction.} that for
$\phi\in L^{2}(0,T;W^{1,q}_{0,\di}(\Omega_{\eta}))$ we find that
\begin{align}
\label{eq:weak-time}
\int_0^T\skp{\partial_tv+[\nabla v] v}{\phi}_\eta\, dt=-\int_0^T\int_{\Omega_{\eta}}\nabla v\cdot \nabla \phi\, dx\, dt.
\end{align}

Moreover, $(\partial_t v+[\nabla v] v,\partial_t^2\eta) \in L^2(0,T;\mathcal{W}^*)$ for 
\[
\mathcal{W}=\{(\phi,b)\in W^{1,q}_{\di}(\Omega_{\eta})\times H^2({\omega})\, :\,\phi(t,x,\eta(x))=b(t,x)\}
\]
for any $q\in (2,\infty)$ if $s=2$ and $q=2$ if $s>2$.

 In particular, for all 
$(\phi,b)\in \mathcal{W}$ we find that
\begin{align*}
\int_0^T\skp{\partial_tv+[\nabla v] v}{\phi}_\eta+\skp{\partial_t^2\eta}{b}\, dt=-\int_0^T\int_{\Omega_{\eta}}\nabla v\cdot \nabla \phi \, dx\, dt +\int_0^T\int_{{\omega}} \nabla^2 \eta \cdot \nabla^2b\, dx\, dt.
\end{align*}
\end{prop}
\begin{proof}
Let $\phi\in L^{2}(0,T;W^{1,q}_{0,\di}(\Omega_{\eta}))$.
First observe, that if (additionally) $\partial_t\phi\in L^2([0,T]\times \Omega_\eta)$ and $\nabla \phi\in L^\infty(0,T;L^2(\Omega_\eta))$, than (as $\abs{v}^2\in L^1_t(L^2_z)$) we find
\begin{align*}
\int_0^T\skp{\partial_tv+(v\cdot\nabla)v}{\phi}_\eta:&=\int_{\Omega_{\eta(T)}} v(T)\cdot \phi(T)\, dz-\int_{\Omega_{\eta_0}}v^0\cdot \phi(0)\, dz-\int_0^T\int_{\Omega_{\eta}}v\cdot\partial_t\phi + v\otimes v\cdot\nabla\phi\,dz\, dt 
\\
&=-\int_0^T\int_{\Omega_{\eta}}\nabla v\cdot \nabla \phi\, dz\, dt.
\end{align*}
Hence, by taking the mollification introduced in Lemma~\ref{lem:molly} (here $b\equiv0$), we find that
\begin{align*}
\int_0^T\skp{\partial_tv+(v\cdot\nabla)v}{\phi_\delta}_\eta
&=-\int_0^T\int_{\Omega_{\eta}}\nabla v\cdot \nabla \phi_\delta\, dz\, dt,
\end{align*}
which implies the result by passing with $\delta\to 0$ by the convergence result of Lemma~\ref{lem:molly}. This allows to give the left hand side a well defined meaning; hence the domain of the left hand side can accordingly be extended. The proof of the second identity is analogous.

\end{proof}

\section{Proof of the main result}

\subsection{The set-up}
Throughout this section let $(v_1,\eta_1)$, $(v_2,\eta_2)$ be weak solutions to \textit{FSI} for initial conditions $v_1(0)=v_{1,0}$, $v_2(0)=v_{2,0}$, $\eta_1(0)=\eta_{1,0}$ $\eta_2(0)=\eta_{2,0}$ and $\partial_t\eta_1(0)=\eta_{1,0}^*$, $=\partial_t\eta_2(0)= \eta_{2,0}^* $. Let $v_2$ satisfy the additional regularity assumption  $v_2 \in L^{r}(0,T;W^{1,s}(\Omega_{\eta_2}))$, $\partial_t v_2 \in L^2(0,T;W^{-1,r}(\Omega_{\eta_2}))$  for some  $s>3$, $r>2$.  Note that as $\partial_t \eta_1=\Tr_{\eta_1}(v_1)$ and $\partial_t\eta_2 =\Tr_{\eta_2}(v_2)$ we have by the trace theorem for moving boundaries (see \cite[Lemma~6]{BreSch18}])
$$\partial_t\eta_1 \in L^2(0,T;H^{l}(\omega)), \quad \partial_t\eta_2 \in L^{r}(0,T;W^{\frac{3}{2},3}(\omega))$$
for all $l \in (0,1/2)$. By Theorem~\ref{thm:boris} we find additionally that
$$ \eta_1 \in L^2(0,T;H^{2+l}(\omega)),\quad \eta_2 \in L^r(0,T;H^{2+l}(\omega)), ~~l \in (0,1/2).$$ 

We
define
 the variable in time domains
\[
\Omega_1:=\Omega_{\eta_1}\text{ and }\Omega_2:=\Omega_{\eta_2}.
\]
Since most of the computations will be given on the domain of the weak solution $\Omega_1$ we introduce for $u:[0,T]\times \Omega_1\to \R^3$ the notation
\begin{align*}
\norm{u(t)}_{k,p}:=\norm{u(t)}_{W^{k,p}(\Omega_{\eta_1(t)})}, \quad \norm{u(t)}:=\norm{u(t)}_{L^2(\Omega_{\eta_1(t)})}\text{ and } (u(t),w(t)):=\skp{u(t)}{w(t)}_{\eta_1},
\end{align*}
whenever well defined.
Recall also, that in case a function $b:[0,T]\times \omega\to \R$ we will extend it constantly to  $b:[0,T]\times \omega\times \R \to \R$ without further notice. For such function we use
\begin{align*}
\norm{b(t)}_{k,p}:=\norm{b(t)}_{W^{k,p}(\omega)}, \quad \norm{b(t)}:=\norm{b(t)}_{L^2(\omega)}\text{ and } (u(t),w(t)):=\skp{u(t)}{w(t)}_{\omega}.
\end{align*}

The first step of the proof is to introduce a diffeomorphism $\psi: \Omega_1 \to \Omega_2$ to compare the velocity fields on the same domain. We define such a $\psi$ explicitly by 
\begin{align*}
\gamma:{\omega} &\to (0,\infty), \quad x \mapsto \frac{\eta_2(x)}{\eta_1(x)},\\ \psi: [0,T]\times {\omega} \times \R &\to [0,T] \times {\omega} \times \R \quad (t,x,y)\mapsto (t,x,\gamma(t,x)y).
\end{align*}
Then $\psi(\{t\}\times\Omega_1)=\{t\}\times \Omega_2$ for all $t \in [0,T]$. Note however that this transformation does not conserve the property of vanishing divergence. For that we follow the approach in~\cite{Gui10}. Define the $3\times 3$ matrix\footnote{ Here and in the following we use $(\mathbb{I}_2, 0)$ for
$\begin{pmatrix}
1 & 0 & 0
\\
0 & 1 & 0
\end{pmatrix}
$.
}
\begin{align*}J(t,x,y)&=D_z\psi(t,x,y)=\begin{pmatrix}
\mathbb{I}_2 & 0\\
y\nabla \gamma(t,x) & \gamma(t,x)
\end{pmatrix},\\ 
\tilde{J}&=J\circ \psi^{-1}=\begin{pmatrix}
\mathbb{I}_2&0\\
y\gamma^{-1}\nabla\gamma&\gamma(t,x)
\end{pmatrix}.
\end{align*}
Now for $w: [0,T] \times \Omega_2 \to \R^3$ set $\hat{w}=\gamma J^{-1} (w \circ \psi)$ and for $u: [0,T] \times \Omega_1 \to \R^3$ set $\check{u}=\gamma^{-1}\tilde{J} u \circ \psi^{-1}$. The next lemma shows that $(\hat{w}, \xi)$ is an admissible and solenoidal test function for $(v_1, \eta_1)$ if $(w, \xi)$ is an admissible and solenoidal  test function  for $(v_2, \eta_2)$ and  $(\check{u},\xi)$ is an admissible and solenoidal test function for $(v_1,\eta_1)$  if $(u,\xi)$ is an admissible and solenoidal for $(v_2,\eta_2)$.

\begin{lem}
	\label{testf}
	Let $w\in  L^1(0,T;W^{1,q}(\Omega_2; \R^3)), u: [0,T] \to \Omega_1$ (sufficiently smooth). The following holds
	\begin{enumerate}
		\item 
		If $\di w=\di u=0$ then $\di \hat{w}=\di \check{u}=0$. 
		\item
		$u^3(t,x,\eta_2(t,x))=\hat{u}^3(t,x,\eta_1(x))$, $u^3(t,x, \eta_1(x)) = \check{u}^3(t,x,\eta_2(x))$.
		\item 
		$(u-\hat{w})\circ \psi^{-1}=\gamma\tilde{J}^{-1}(\check{u}-w)$ and $(\check{u}-w) \circ \psi=\gamma^{-1} J(u-\hat{w})$
	\end{enumerate}
\end{lem}

\begin{proof}
	We calculate
	$$\gamma J^{-1}=\begin{pmatrix}
	\gamma \mathbb{I}_2 & 0\\
	-y \nabla \gamma & 1
	\end{pmatrix}, \quad \gamma^{-1}\tilde{J}=\begin{pmatrix}
	\gamma^{-1} \mathbb{I}_2 & 0\\
	y\gamma^{-2}\nabla \gamma & 1
	\end{pmatrix}=\begin{pmatrix}
	\gamma^{-1} \mathbb{I}_2&0\\
	-y\nabla (\gamma^{-1})&1
	\end{pmatrix}. $$
	Thus it is sufficient to prove (1) and (2) for $\hat{w}$ as for $\check{u}$ we just have to replace $\gamma$ by $\gamma^{-1}$ everywhere. We get
	$$\hat{w}=(\gamma w' \circ \psi, -y \nabla \gamma \cdot w'  \circ \psi+w^2\circ \psi),$$
	As $\psi(x, \eta_1)=(x,\eta_2)$ this directly yields the second assertion.
		For the divergence we find
	$$\di_x \hat{w}' =\nabla \gamma \cdot w'  \circ \psi+\gamma \di_x (w'  \circ \psi)=\nabla \gamma \cdot w'  \circ \psi+\gamma((\di_x w' ) \circ \psi+(\partial_y w' ) \circ \psi) \cdot y\nabla \gamma)$$
	and using $\partial_y(w \circ \psi)=\gamma (\partial_y w) \circ \psi$
	$$\partial_y \hat{w}^2=-\nabla \gamma \cdot w' \circ \psi+\gamma(-y\nabla\gamma \cdot (\partial_y w' ) \circ \psi+(\partial_y w^2) \circ \psi).$$
	Thus $\di w_1=0$ gives
	$\di \hat{w}=\gamma (\di_x w) \circ \psi=0$.
	For (3) note first that
	$$
	J^{-1} \circ \psi^{-1}=\begin{pmatrix}
	\mathbb{I}_2&0\\
	-y\gamma^{-2}\nabla \gamma&\gamma^{-1}
	\end{pmatrix}=\tilde{J}^{-1}$$
	This gives
	$$
	(u-\hat{w})\circ \psi^{-1}=u\circ \psi^{-1}-\gamma (J^{-1} \circ \psi^{-1}) w=\gamma\tilde{J}^{-1}(\gamma^{-1}\tilde{J}u\circ \psi^{-1}-w)=\gamma \tilde{J}^{-1}(\check{u}-w).$$
	Lastly
	$$(\check{u}-w)\circ \psi=\gamma^{-1}J u-w \circ \psi=\gamma^{-1}J((u-\hat{w})).$$
\end{proof}
For notational purposes set
\begin{align*}\eta_1-\eta_2&=\eta ,\quad w_1=v_1-\hat{v}_2, \quad w_2=\check{v}_1-v_2.\\
v_2 \circ \psi&=\tilde{v}_2, \quad v_1 \circ \psi^{-1}=\tilde{v}_1, \quad w_2 \circ \psi=\tilde{w}_2, \quad w_1 \circ \psi^{-1}=\tilde{w}_1, \quad \tilde{f}_2=f_2 \circ \psi
\end{align*}
Note that by Lemma \ref{testf}
\begin{equation}
\label{eq:what}
\tilde{w}_2=\gamma^{-1}Jw_1, \quad \tilde{w}_1=\gamma \tilde{J}^{-1}w_2,
\end{equation}
and with a slight missuse of notation. 
$$
\check{v}_{1,\delta}=\gamma^{-1}\tilde{J}v_{1,\delta} \circ \psi^{-1},\quad \hat{v}_{2,\delta}=\gamma J^{-1}v_{2,\delta}\circ \psi, \quad w_{2,\delta}=\check{v}_{1,\delta}-v_{2,\delta}, \quad w_{1,\delta}=v_{1,\delta}-\hat{v}_{2,\delta}.$$

Note that by Lemma \ref{lem:molly} $\di v_{2,\delta}=\di v_{1,\delta}=0$ and $v_{2,\delta}(x,\eta_2(x))=(0,\partial_t \eta_{2,\delta})$, $v_{1,\delta}=(0,\partial_t \eta_{1,\delta})$. Thus by Lemma \ref{testf} $\di \hat{v}_{2,\delta}=\di \check{v}_{1,\delta}=0$ and $\hat{v}_{2,\delta}(x, \eta_1(x))=\partial_t\eta_{2,\delta}$, $\check{v}_1(x,\eta_2(x))=\partial_t\eta_{1,\delta}$ as well as $\di w_{1,\delta}=\di w_{2,\delta}=0$ and $w_{1,\delta}(x,\eta_1(x))=w_{2,\delta}(x,\eta_2(x))=\partial_t \eta_\delta$.

\subsection{A-priori estimates}
Before we turn to the main argument we collect some results that show that our test-functions are admissible and that the error terms due to the geometric convolution in time are converging to 0.

\begin{rem}
	\label{rem1}
	The following estimates we will use frequently in the following. They are consequences of H\"older's inequality and the imbeddings $H^1(\omega)\hookrightarrow L^p(\omega)$ ($p \in [1,\infty)$) and in case $q<3$, that $W^{1,q}(\Omega_i) \hookrightarrow L^r(\Omega_i)$ for all $r<3q/(3-q)$) ($i=1,2$ here and in the following). See \cite{LenRuz14} for a reference.
	\begin{enumerate}
		\item 
		For all $s\in (1,\infty)$,  $p \in [1,s)$ and $f \in L^{s}(\Omega_i)$, $g \in H^1(\omega)$
		$$
		\|fg\|_{L^p(\Omega_i)} \le C\|f\|_{L^s(\Omega_i)}\|g\|_{H^1(\omega)}.$$
		\item 
		For all $p \in (1,2)$, $q\in(6p/(6-p),3)$, $f \in W^{1,q}(\Omega_i)$ and $g \in L^2(\Omega_i)$ 
		$$\|fg\|_{L^p(\Omega_i)} \le C\|f\|_{W^{1,q}(\Omega_i)}\|g\|_{L^2(\Omega_i)}.$$
		\item 
		If $p,q$, $f$ are as above and $g \in H^2(\omega)$ 1. and 2. give in particular
		$$
		\|fg\|_{W^{1,p}(\Omega_i)} \le C \|f\|_{W^{1,q}(\Omega_i)}\|g\|_{H^2(\omega)}.
		$$
	\end{enumerate}
\end{rem}

\begin{lem}
	\label{liste}
	Let $(v_1, \eta_1), (v_2,\eta_2) \in \mathcal{V}_S$ weak solutions of \textit{FSI}, $(v_2, \eta_2)$ satisfying the additional regularity assumptions. Then
	\begin{enumerate}
		\item
		$\gamma$ satisfies the following estimates for a.e.\ $t\in [0,T]$.
		$$
		\|\gamma(t)-1\|_{H^2(\omega)} \le C\|\eta(t)\|_{H^2(\omega)}
		\quad  \| \partial_t \gamma(t) \|_{L^2(\omega)} \le C\|\partial_t\eta(t)\|_{L^2(\omega)}+C\|\eta(t)\|_{L^2(\omega)}.
		$$
		The same estimates hold for $\gamma^{-1}$.
		\item
		$\nabla \gamma\in L^\infty(0,T;L^q(\omega))$ for all $q\in [1,\infty)$
		\[
		\norm{\nabla \gamma(t)}_{L^q(\omega)}\leq C\|\eta(t)\|_{H^2(\omega)}
		\]
		and the same holds for $\gamma^{-1}$.
		\item
		$\hat{v}_1 \in  L^\infty(0,T; L^p(\Omega_2)) \cap L^2(0,T; W^{1,p}(\Omega_2))$ for all $p \in (1,2)$  and 
		$\|\hat{v}_1\|_{W^{1,p}(\Omega_2)} \le  C \|v_1\|_{1,2} $ for all $p \in [1,2)$.
		\item $\partial_t\hat{v}\in L^2(0,T;\tilde{W}^{-1,p'}(\Omega_1))$ for all $p'\in [1,r)$,
	\end{enumerate} 
\end{lem}

\begin{proof}
	(1) and (2):
	
	\noindent
	It holds
	\begin{align*}
	\gamma-1=&\frac{\eta_2-\eta_1}{\eta_1} \le C|\eta|\\
	\gamma_t&=\frac{\partial_t\eta_2 \eta}{\eta_1^2}-\frac{\eta_2\partial_t\eta}{\eta_1^2} \le C(|\partial_t \eta_2||\eta|+|\partial_t\eta|),\\
	\nabla\gamma&=\frac{\nabla\eta_2\eta}{\eta_1^2}-\frac{\eta_2\nabla\eta}{\eta_1^2} \le C(|\nabla\eta_2||\eta|+|\nabla\eta|),\\
	\partial_{x_ix_j}^2\gamma& = \eta_1^{-2}(\partial_{x_ix_j}^2 \eta_2 \eta+\partial_{x_j}\eta_2\partial_{x_i}\eta-\partial_{x_i}\eta_2\partial_{x_j}\eta-\eta_2\partial_{x_ix_j}^2\eta)-2\frac{\partial_{x_i}\eta_1}{\eta_1^3}\partial_{x_j}\gamma\\
	& \le C(|\nabla^2 \eta_2||\eta|+|\nabla \eta_2||\nabla \eta|+|\nabla^2\eta|+|\nabla \eta_1|(|\nabla\eta_2||\eta|+|\nabla\eta|)
	\end{align*}
(1) and (2) now follow from the embeddings $H^2(\omega) \hookrightarrow W^{1,q}(\omega) \hookrightarrow L^\infty(\omega)$ for all $q \in [1, \infty)$. The results for $\gamma^{-1}$ follow by replacing the roles of $\eta_1$ and $\eta_2$

%
%
	

\newpage
\noindent
Proof of (3):

\noindent
We calculate
	$$
	\partial_{x_i}(\gamma^{-1}\tilde{J})=\begin{pmatrix}
	\mathbb{I}_2 \partial_{x_i}(\gamma^{-1}) &0 \\
	-y\partial_{x_i}\nabla(\gamma^{-1}) & 0
	\end{pmatrix}, \quad
	\partial_y (\gamma^{-1} \tilde{J})=\begin{pmatrix}
	\mathbb{I}_20&  0\\
	-\nabla (\gamma^{-1}) & 0
	\end{pmatrix}
	$$
	Hence
	\begin{equation}
	\label{J}
	|\partial_{x_i}(\gamma^{-1}\tilde{J})|+|\partial_{y}(\gamma^{-1} \tilde{J})| \le C(|\nabla (\gamma^{-1})|+|\nabla (\gamma^{-1})|^2+|y\nabla^2 (\gamma^{-1})|)
	\end{equation}
Observe further, that by Lemma~\ref{lem:convective}
	$\int_{0}^{\eta_1 (t,x)}\abs{v}\, dy \in L^2(0,T;L^q({\omega}))$ for all $q\in [1,\infty)$, which implies (using also (2)) that 
	\begin{align}
	\label{eq:etav1}
	\begin{aligned}
		&\abs{\nabla^2\gamma}\abs{v_1}\in L^2(0,T;L^p(\Omega_1))\text{ and }\abs{\nabla^2\gamma}\abs{\tilde{v}_1}\in L^2(0,T;L^p(\Omega_2))\text{ for all }p\in [1,2)
	\end{aligned}
	\end{align}
	
	Now by \eqref{J}
	$$|\partial_{z_i}(\gamma^{-1}J \tilde{v}_1)| \le  C(|\nabla (\gamma^{-1})|+|\nabla (\gamma^{-1})|+|\nabla^2 (\gamma^{-1})||\tilde{v}_1|+|\nabla (\gamma^{-1})||(\nabla v_1) \circ \psi^{-1}|$$
	Thus the assertion for $\hat{v}_1$ follows using also (1), (2) and Remark \ref{rem1}.
	
	\noindent
	Proof of (4):
	
%

\noindent
This estimate is analogous to (3) in Lemma~\ref{lem:molly}:
Let us take $\psi\in \tilde{W}^{1,p'}({\omega}\times \R)$, such that $\psi(t,x,y)=0$ for all $x\in B_c$ and $\norm{\psi}_{W^{1,p'}([0,T]\times{\omega}\times \R)}\leq 1$ to find that 
	\begin{align*}
	\int_0^T(\partial_t \hat{v}_2,\psi)\, dt
&= \int_0^T(\partial_t(\gamma J^{-1})\tilde{v}_2, \psi)\,dt
 +\int_0^T\skp{J^{-1}\partial_tv_2}{ \psi}_{\eta_2}\, dt
 +\int_0^T\int_{\Omega_1} \gamma J^{-1} \partial_3 v_2 \partial_t \gamma \cdot \psi \, dz\,  dt
	\end{align*}
	The estimates on the first and the third term are now straight forward using the assumptions on $v_2$. In the first term it is important to observe that the terms involving $\partial_t \nabla \gamma$ are always coupled to $v_2'$. Using the fact that $v_2'(t,x,\eta_2(t,x))= 0$ for all $(t,x)\in [0,T]\times \omega$, we may use integration  by parts in $x$ direction and find
	\begin{align*}
	(\partial_t(\gamma J^{-1})\tilde{v}_2, \psi) \leq C\int_{\Omega_1}\abs{\partial_t\gamma}(\abs{\nabla \gamma}\abs{\nabla \tilde{v}_2}\abs{\psi}+\norm{v_2}_{L^\infty(\Omega_2)}\abs{\tilde{v}_2}\abs{\nabla \psi}),
	\end{align*}
But these expression can be estimated using that $p^*=\frac{3p}{3-p}$ can be assumed to be close enough to 6 such that
\begin{align*}
&\int_0^T(\partial_t(\gamma J^{-1})\tilde{v}_2, \psi)\,dt
\\
&\quad  \leq C\int_0^T\norm{\partial_t\gamma}
\big(\norm{\abs{\nabla \tilde{v}_2}\abs{\nabla \gamma}}_{3+(3-s)/2}\norm{\psi}_{{p^*}}+\norm{v_2}_{W^{1,s}(\Omega_2)}\norm{\psi}_{{1,p}}\big)\, dt.
\end{align*}
This expression is bounded since $\partial_t\eta\in L^\infty(L^2)$ and $\abs{\nabla \gamma}\abs{\nabla \tilde{v}_2}\in L^2{(0,T;L^q(\Omega_1))}$ for all $q\in [3,s)$.
\end{proof}

At this point we choose $t\in [0,T]$ such that all involved quantities do have a Lebesgue point at this time instance. Without any further notice we extend all quantities via \eqref{extend} constant on $(-\infty,0]$ and $[t,\infty)$.

Next we take the convolution introduced in Lemma~\ref{lem:molly} on $w_2$ and $\hat{v}_2$. We will need the following convergences:
\begin{lem}
	\label{liste3}
	The following expressions are all well defined and convergence to zero with $\delta \to 0$:
	\begin{align}
	\label{con1}
	&\int_0^t \skp{\partial_t v_2}{w_2-w_{2,\delta}}_{\eta_2}+\skp{ [\nabla v_2] v_2}{ w_2-w_{2,\delta}}_{\eta_2}+\skp{\nablasym v_2}{ \nablasym w_2-\nablasym w_{2,\delta})}_{\eta_2}~dt0\\
	\label{con2}
	&{\int_0^t (v_1 \otimes v_1,\nabla \hat{v}_{2}-\nabla \hat{v}_{2,\delta})~dt }
	\\
	\label{con3}
		&{(v_1(t),\hat{v}_2(t)-\hat{v}_{2,\delta}(t))-\int_0^t (v_1,\partial_t \hat{v}_2-\partial_t\hat{v}_{2,\delta})- (\nablasym v_1,\nablasym \hat{v}_2-\nablasym \hat{v}_{2,\delta})~dt }.
	\end{align}
	Moreover, $(\partial_t\eta_{\delta},\hat{v}_{2,\delta})$ is a valid testfunction for the weak formulation of $(\eta_1,v_1)$ and the terms

\noindent
$\skp{\partial_t v_2}{w_{2,\delta}}_{\eta_2}$, $\skp{\nablasym v_2}{\nablasym w_{2,\delta}}_{\eta_2}$, $\skp{[\nabla v_2] v_2 }{w_{2,\delta}}_{\eta_2} \in L^1(0,T)$ uniformly in $\delta$.
\end{lem}
\begin{proof}
For \eqref{con1} we know that $w_2 \in L^2(0,T;W^{1,p}(\Omega_2)$ for all $p \in [1,2)$ by Lemma~\ref{liste}. Hence by Lemma~\ref{lem:molly} $w_2-w_{2,\delta} \to 0$ weakly in $L^2(0,T;W^{1,p}(\Omega_2))$ for all $p \in [1,2)$. Since it is a valid argument for $\partial_t v_2 \in L^2(0,T;\tilde{W}^{-1,p'}(\Omega_2))$ and since $\nabla v_2 \in L^2(0,T;W^{1,s}(\Omega_2))$ for $s>3$ it yields the convergence of the first and third term. Moreover, it was shown in Lemma~\ref{liste} (6) that  $[\nabla v_2]v_2 \in L^2(0,T;L^q(\Omega_2)$ for some $q>(6/5)$. Since we may assume $p\in [1,2)$ such that $W^{1,p}(\Omega_2) \hookrightarrow L^{q'}$ the convergence of the second term follows again from the weak convergence of $w_{2,\delta}$ in $L^2(0,T;W^{1,p}(\Omega_2))$. 
	
	 In \eqref{con2} we will show that all involved terms are uniformly bounded. The uniform bounds imply that all weakly converging sub-sequences converge to 0, by the uniqueness of the weak limits. The critical term here is
	 $\int_0^T\int_{\Omega_{\eta_1}}\abs{v_1\otimes v_1\cdot\nabla(\partial_{x_i} \gamma \tilde{v}_{2,\delta})}\,dz\, dt $. All other terms can be estimated in a straight forward manner and we skip the details.
Using the uniform bounds on $\eta_1,\eta_2,\frac{1}{\eta_1},\frac{1}{\eta_2}$ we find 
 \begin{align*}
&\int_{\Omega_1}\abs{v_1\otimes v_1\cdot\nabla(\partial_{x_i} \gamma \tilde{v}_2)}\,dz\, dt 
\\
&\quad \leq C
\int_{{\omega}}\int_{0}^{\eta_1(t,x)}\abs{v_1}^2\abs{\tilde{v_2}}\, dy\Big((\abs{\nabla \eta_1}+\abs{\nabla \eta_2})(1+\abs{\nabla \eta_2}+\abs{\nabla^2\eta_2})+\abs{\nabla \eta_2}\abs{\nabla^2\eta_1}\Big)\, dx
\\
&\quad  + C\int_{\Omega_1}\abs{v_1}^2\abs{\nabla\tilde{v_2}}\, dy\Big(1+\abs{\nabla\eta_1}^2+\abs{\nabla\eta_2}^2\Big)\, dz
=: I_1+I_2.
\end{align*}
Using Lemma~\ref{lem:convective} and H\"older's inequality in space we can estimate
\begin{align*}
I_1& \leq C \norm{v_2}_{L^\infty(\Omega_{\eta_2})}\int_{{\omega}}\int_{0}^{\eta_1(t,x)}\abs{v_1}^2\, dy(\abs{\nabla \eta_1}+\abs{\nabla \eta_2})\Big(1+\abs{\nabla^2\eta_1}+\abs{\nabla^2\eta_2}\Big)\, dx 
\\
& \leq C\norm{v_2}_{L^\infty(\Omega_{\eta_2})}(\norm{\eta_1}_{1,\infty}+\norm{\eta_2}_{1,\infty})(\norm{\eta_1}_{2,2}
+\norm{\eta_2}_{2,2}+1)\bigg\|\int_{0}^{\eta_1(t,x)}\abs{v_1}^2\, dy\bigg\|\\
& \le C \norm{v_2}_{L^\infty(\Omega_{\eta_2})}(\norm{\eta_1}_{1,\infty}+\norm{\eta_2}_{1,\infty})(\norm{v_1}^2+\norm{v_1}\norm{\nabla v_1}+\|\partial_t \eta_1\|\|\nabla \eta_1\|_{\infty})\\
& \le C (\|v_2\|_{W^{1,s}(\Omega_2)}+1)^2(\|\eta_1\|_{1,\infty}+\|\eta_2\|_{1,\infty}+1)^2(\|v_1\|^2+\|\partial_t \eta_1\|^2)+C \| v_1\|_{1,2}^2
 \end{align*}
 Since $v_2\in L^{r}(0,T;W^{1,s}(\Omega_{\eta_2}))$ for some $r>2$ and $\eta_1, \eta_2\in L^q(0,T;W^{1,\infty}(\omega))$ for all $q<\infty$ (Theorem~\ref{thm:boris})   $ \norm{v_2}_{L^\infty(\Omega_{\eta_2})}(\norm{\eta_1}_{1,\infty}+\norm{\eta_1}_{1,\infty} \in L^2([0,T])$. As additionally $v_1 \in L^\infty(0,T;L^2(\Omega_1)) \cap L^2(0,T,H^2(\Omega_1))$ and $\partial_t \eta_1 \in L^\infty(0,T;L^2(\omega))$ the last term is bounded in time.

To estimate $I_2$ note that as $v_1 \in L^2(0,T;H^1(\Omega_1)) \hookrightarrow
 L^2(0,T;L^\alpha(\Omega_1))$ for all $ \alpha \in [1,6)$ we find for all $a<3/2$ (i.e. $(\frac{2}{a})'<4$)
$$\|v_1\|_a \le \|v_1\|_2\|v_1\|_{(\frac{2}{a})'} \le \|v_1\|\|v_1\|_{1,2}$$
Now choose $p>1$, $q>3$ such that $qp <s$ and $pq'<3/2$.
\begin{align*}
I_2&\leq C (1+\|\nabla \eta_1\|_{2p'}+\|\nabla\eta_2\|_{2p'})\||\nabla \tilde{v}_2||v_1|^2\|_p\\
& \le C(\|\eta_1\|_{2,2}+\|\eta_2\|_{2,2})
\|\nabla \tilde{v}_2\|_{pq} \||v_1|^2\|_{pq'}
\le C \|v_2\|_{W^{1,s}(\Omega_2)} \|v_1\| \|v_1\|_{1,2}
\end{align*}
which is bounded in time due to the regularities on $v_2$ and $v_1$.
	We continue with \eqref{con3}. We write
	\begin{align*}
	\int_0^t (v_1,\partial_t \hat{v}_2-\partial_t\hat{v}_{2,\delta})\, dt
	&= \int_0^t (\gamma J^{-T} v_1, \partial_t \tilde{v}_{2}-\partial_t\tilde{v_{2,\delta}})\, dt+ \int_0^t (v_1, \partial_t (\gamma J^{-1}) (\tilde{v}_{2}-\tilde{v}_{2,\delta}))\, dt
	\\
	&=\int_0^t \skp{ J^{-T} \tilde{v}_1}{\partial_t v_{2}-\partial_t v_{2,\delta}}_{\eta_2}\, dt + \sum_{i=1}^2\int_0^t (v_1^i, \partial_t \gamma  \tilde{v}_{2}^i-\tilde{v}_{2,\delta}^i))\, dt 
	\\
	&\quad - \sum_{i=1}^2\int_0^t (y\partial_{x_i}\partial_t\gamma, v_1^i(\tilde{v}_{2}^3-\tilde{v}_{2,\delta}^3))\, dt
=:(i)+(ii)+(iii)
	\end{align*}
The term $(i)$ converges to 0 by Lemma~\ref{lem:molly} using that 
by an analogous estimate to Lemma~\ref{liste}, (3) we find that $ J^{-T}\tilde{v}_1\in L^2(W^{1,p}(\Omega_2)$ for all $p\in (1,2)$. The term $(ii)$ converges directly by Lemma~\ref{lem:molly} and Lemma~\ref{liste}. On the term $(iii)$ we integrate by parts  to find that
\[
\abs{(iii)}\leq \int_0^t\int_{\Omega_1}\abs{\partial_t\gamma}\abs{\nabla (v_1(\tilde{v}_2^3-\tilde{v}_{2,\delta}^3)}\, dz\, dt
\]
which can be bounded uniformly (using Lemma~\ref{lem:molly} and Lemma~\ref{liste} again) and therefore converges to 0.
The estimate on the part involving symmetric gradients is straight forward using the bounds in Lemma~\ref{lem:molly} and Lemma~\ref{liste}.
It remains to show that the first term in \eqref{con3} converges. For that we simply use the fact that we chose $t$ to be a Lebesgue point of all involved quantities. Hence by the very definition of $\mol$, we find that 
	\begin{align*}
	\lim_{\delta\to 0}(v_1(t),\hat{v}_{2,\delta}(t))=(v_1(t),\hat{v}_{2}(t)).
	\end{align*}
For the last statement observe that for all $p\in [1,2)$ by the calculations in Lemma~\ref{liste} that $w_2=v_2-\hat{v}_1\in L^2(0,T;W^{1,p}(\Omega_2))$ and therefore by Lemma~\ref{lem:molly} $w_{2,\delta}\in L^2(0,T;W^{1,p}(\Omega_2))$. This holds in particular for $p=r'$ which yields that the first two terms are in $L^1(0,T)$. Further, since $\nabla v_2 \in L^2(0,T;L^s(\Omega_2))$ for  $s>3$ H\"older's inequality implies for some $q>\frac65$ 
	$$\|[\nabla v_2] v_2 \|_{L^q(\Omega_2)}\le \|v_2\|_{L^2(\Omega_2)} \|\nabla v_2\|_{L^{2q'}(\Omega_2)}.$$
	Choosing $q>6/5$ such that $(2/q')<s$ bounds the right hand side in $L^2([0,T])$. As by embedding  $w_{2,\delta} \in L^2(0,T;L^a(\Omega_2)$ for all $a \in [1,6)$ we find that $[\nabla v_2] v_2\cdot w_2\in L^1(0,T;L^1(\Omega_2))$.	

\end{proof}

\subsection{The stability estimate (Proof of Theorem~\ref{the:2})}
We have collected all the necessary notations and estimates to start the stability estimate. The estimate is derived by testing first the equation of $(v_2,\eta_2)$ by $(w_{2,\delta},\partial_t \eta_\delta)$, second the energy inequality for $(v_1,\eta_1)$ and finally testing $(v_1,\eta_1)$ with $(\mol,\partial_t \eta_{2,\delta})$.

Testing the equation of $(v_2,\eta_2)$ by $(w_{2,\delta},\partial_t \eta_\delta)$, integration by parts and Reynold's transport theorem give
\begin{align} \begin{aligned} 
\label{eqv1}
&\int_0^t\skp{\partial_t v_2+[\nabla v_2] v_2}{w_{2,\delta}}_{\eta_2}+\skp{\nablasym v_2}{\nablasym w_{2,\delta}}_{\eta_2} -\skp{f_2}{w_{2,\delta}}_{\eta_2}~dt\\
&\quad +(\partial_t \eta_2,\partial_t \eta_\delta)-(\partial_t \eta_{2,0},\partial_t \eta_0)-\int_0^t(\partial_{t}\eta_2,\partial_t^2\eta_{\delta})-(\Delta \eta_2, \Delta \partial_t \eta_\delta)-(g_2, \partial_t \eta_{\delta})dt=0.
\end{aligned} 
\end{align}

We can write this 
\begin{align} \begin{aligned}
\label{eqv22}
&\int_0^t\skp{\partial_t v_2+[\nabla v_2] v_2
}{w_{2}}_{\eta_2}+\skp{\nablasym v_2}{\nablasym w_{2}}_{\eta_2} -\skp{f_2}{w_{2}}_{\eta_2}~dt
\\
&\quad +(\partial_t \eta_2,\partial_t \eta_\delta)-(\partial_t \eta_{2,0},\partial_t \eta_0)-\int_0^t(\partial_{t}\eta_2,\partial_t^2\eta_{\delta})-(\Delta \eta_2, \Delta \partial_t \eta_\delta)- (g_2,\partial_t \eta_{\delta})dt
= K_{1,\delta}
\end{aligned} \end{align}
where 
\begin{equation*}
K_{1\,\delta}:=\int_0^t\skp{\partial_t v_2+[\nabla v_2]v_2 }{w_2-w_{2,\delta}}_{\eta_2}+\skp{\nablasym v_2}{\nablasym (w_2-w_{2,\delta})}_{\eta_2}-\skp{f_2}{w_2-w_{2,\delta}}_{\eta_2}
\end{equation*}
Then $K_{1,\delta} \to 0$ for $\delta \to 0$ by Lemma~\ref{liste3}.

The next step is to transform the equation for $v_2, \eta_2$ to the domain $\Omega_1$. In particular we want to prove an estimate for
$$
\int_0^t  (\partial_t\hat{v}_2+\nabla \hat{v}_2\hat{v}_2, w_1)+(\nabla \hat{v}_2,\nabla w_1)-(\tilde{f}_2, w_1)\, dt
$$
First compute
\begin{align*}
&(\partial_t\hat{v}_2, w_1)=(\gamma J^{-1}\partial_t \tilde{v}_2+\partial_t(\gamma J^{-1})\tilde{v}_2), w_1))
\\
&=\skp{ \tilde{J}^{-1}((\partial_t \tilde{v}_2)\circ \psi^{-1})}{\tilde{w}_1}_{\eta_2}+(\partial_t(\gamma J^{-1}) \tilde{v}_2, w_1).
\end{align*}

By chain rule we get
$$(\partial_t\tilde{v}_2)\circ \psi^{-1}=\partial_t v_2+y\gamma^{-1}\partial_t \gamma\partial_y v_2,$$
Also using $w_2=\gamma^{-1}\tilde{J}\tilde{w_1}$ (cf. \eqref{eq:what}) this gives
\begin{align*}
\tilde{J}^{-1}(\partial_t \tilde{v}_2)\circ \psi^{-1} \cdot \tilde{w}_1
=\partial_t v_2 \cdot w_2+\partial_t v_2 \cdot (\tilde{J}^{-t}\tilde{w_1}-w_2)+y\gamma^{-1}\partial_t \gamma\tilde{J}^{-1}\partial_y v_2 \cdot \tilde{w}_1\\
=\partial_t v_2 \cdot w_2+\partial_t v_2 \cdot (\tilde{J}^{-t}-\gamma^{-1}\tilde{J})\tilde{w_1}+y\gamma^{-1}\partial_t \gamma\tilde{J}^{-1}\partial_y v_2 \cdot \tilde{w}_1,
\end{align*}

which yields
\begin{align} \begin{aligned}
\label{ineq1}
&\skp{\partial_t v_2 }{ w_2}_{\eta_2}=(\partial_t\hat{v}_2, w_1)-(\partial_t(\gamma J^{-1}) \tilde{v}_2, w_1)
-\skp{\partial_t v_2}{(\tilde{J}^{-t}-\gamma^{-1}\tilde{J})\tilde{w_1}}_{\eta_2}
\\
&\quad
+\skp{y\gamma^{-1}\partial_t \gamma\tilde{J}^{-1}\partial_y v_2} {\tilde{w}_1}_{\eta_2}=:(\partial_t\hat{v}_2, w_1) +R_1.
\end{aligned} 
\end{align}
\begin{proof}[Estimate of $R_1$]
	With similar estimates as in the proof of Lemma \ref{liste} we get
	\begin{equation}
	\label{tiJ}
	| \tilde{J}^{-t}-\gamma^{-1} \tilde{J}| \le C(|1-\gamma|+|\nabla \gamma|), \quad | \nabla(\tilde{J}^{-t}-\gamma^{-1} \tilde{J})| \le C(|\nabla \gamma|+|\nabla^2 \gamma|)
	\end{equation}
	Hence as in the proof of Lemma~\ref{liste} (1) we have (using also Lemma \ref{liste} (1)) 
	$$\|(\tilde{J}^{-t}-\gamma^{-1}\tilde{J})\tilde{w}_1\|_{W^{1,q}(\Omega_2)} \le C\|\eta\|_{2,2}\|w_1\|_{1,2}$$
	for all $q \in [1,2)$. This yields for $p'\in(2,r]$
	\begin{align*}\
	\skp{ \partial_tv_2 }{ (\tilde{J}^{-t}-\gamma^{-1}\tilde{J})\tilde{w}_1}_{\eta_2}  \le \|\partial_t v_2\|_{\tilde{W}^{-1,p'}(\Omega_2)}\|(\tilde{J}^{-t}-\gamma^{-1}\tilde{J})\tilde{w}_1\|_{W^{1,p}(\Omega_2)}\\
	\le C_\epsilon\|\partial_t v_2\|_{-1,r}^2 \|\eta\|_{2,2}^2 +\epsilon \|w_1\|_{1,2}^2.
	\end{align*}
	By Remark \ref{rem1} we have for $p \in (1,3/2)$, $q \in (p,3/2)$ and $a \in (6q/(6-q),2)$
	$$\||\partial_t \gamma||\nabla \gamma||\tilde{w}_1|\|_{L^p(\Omega_2)} \le \|\nabla \gamma\|_{1,2} \norm{\partial_t \gamma}\norm{\tilde{w}_1}_{L^q(\Omega_2)} \le \|\nabla \gamma\|_{1,2} ~\|\partial_t \gamma\|_2 \|\tilde{w}_1\|_{W^{1,a}(\Omega_2)}
	$$
	Thus by \eqref{tiJ} and Lemma \ref{liste}, we get for $p=s' \in (1,3/2)$
	\begin{align*}
	\skp{y\gamma^{-1}\partial_t\gamma\tilde{J}^{-1}\partial_yv_2}{\tilde{w}_1}_{\eta_2} 
	&\le C\|v_2\|_{1,s} \|~ \||\partial_t \gamma||\nabla \gamma||\tilde{w}_1||\|_{1,p}
	\\
	&\le C_\epsilon\|\partial_t \eta\|^2_2\|\eta\|_{2,2} \|v_2\|_{W^{1,s}(\Omega_2)}^2+\epsilon\|w_1\|_{1,2}^2.
	\end{align*}
	Next compute
	$$
	\partial_t(\gamma J^{-1})=\begin{pmatrix}
	\partial_t \gamma & 0 \\
	-y\partial_t\nabla\gamma &0.
	\end{pmatrix}.
	$$
	By H\"older's inequality we get for all $p \in (3,s)$ and $q=2(p/2)'<6$
	$$\||\nabla \tilde{v}_2|| w_1|\| \le \|\nabla \tilde{v}_2\|_{p} \|w_1\|_{q} \le \|v_2\|_{W^{1,s}(\Omega_2)}\|w_1\|_{1,2},$$
	also
	$$\||\tilde{v}_2||\nabla w_1|\|_2 \le \|\tilde{v}_2\|_\infty\|w_1\|_{1,2} \le \|v_2\|_{W^{1,s}(\Omega_2)}\|w_1\|_{1,2}$$
	This yields
	\begin{align*} (\partial_t(\gamma J^{-1}), \tilde{v}_2, w_1) 
	&\le C\|\partial_t \gamma\| \|\tilde{v}_2\|_\infty \|w_1\|_{1,2}+\|\partial_t \nabla\gamma\|_{-1,2}\|\tilde{v}_2^{1}w_1^2\|_{1,2}
	\\ 
	&\le
	C_\epsilon\|\partial_t \eta\|^2 \|v_2\|_{W^{1,s}(\Omega_2)}^2+\epsilon \|w_1\|_{1,2}^2.
	\end{align*}
	In conclusion
	\begin{equation}
	\label{ineq11}
	|R_1| \le C_\epsilon(\|\eta\|_{2,2}^2+\|\partial_t \eta\|^2)(\|v_2\|_{W^{1,s}(\Omega_2)}^2+\|\partial_tv_2\|_{\tilde{W}^{-1,r}(\Omega_2)}^2)+\epsilon \|w_1\|_{1,2}^2.
	\end{equation}
\end{proof}
To symplify Notation in the next step, for a Matrix $A \in \R^{3\times 3}$ we denote the symmetric part of it as $A^{s}=\frac{1}{2}(A+A^t)$. We get by transformation and chain rule
$$
\skp{\nablasym v_2}{\nablasym w_2}_{\eta_2}=(\gamma (\nabla \tilde{v_2}J^{-1})^s,(\nabla\tilde{w}_2J^{-1})^s)
$$
By \eqref{eq:what} 
\begin{align*}
\gamma\nabla\tilde{w}_2J^{-1}&=\gamma\nabla(\gamma^{-1}Jw_1)J^{-1}
=J\nabla w_1J^{-1}+\gamma\nabla(\gamma^{-1}J)w_1J^{-1}\\
&=\nabla w_1 +\nabla w_1 (J^{-1}-I) +(J-I)\nabla w_1J^{-1}+\gamma\nabla(\gamma^{-1}J)w_1J^{-1}.
\end{align*}
and using $\hat{v}_2=\gamma J^{-1}v_2$
\begin{align*}
\nabla \tilde{v}_2 J^{-1}=\nabla\tilde{v}_2+\nabla\tilde{v}_2(J^{-1}-I)
&=\nabla\hat{v}_2+\nabla((I-\gamma J^{-1})\tilde{v}_2)+\nabla\tilde{v}_2(J^{-1}-I)
\\
&=\nabla \hat{v}_2+(I-\gamma J^{-1})\nabla\tilde{v}_2-\nabla(\gamma J^{-1})\tilde{v}_2+\nabla\tilde{v}_2(J^{-1}-I)
\end{align*}
Hence
\begin{align} \begin{aligned}
\label{ineq2}
&\left(\gamma(\nabla\tilde{v}_2J^{-1})^s:(\nabla \tilde{w}_2J^{-1})^s\right)\\
&=\left((\nabla\tilde{v}_2J^{-1})^s,\nablasym w_1 +\left[\nabla w_1 (J^{-1}-I) +(J-I)\nabla w_1J^{-1}+\gamma\nabla(\gamma^{-1}J)w_1J^{-1}\right]^s\right)
\\
&=(\nablasym\hat{v}_2,\nablasym w_1)+\left((\nabla\tilde{v}_2J^{-1})^s, \left[\nabla w_1 (J^{-1}-I) +(J-I)\nabla w_1J^{-1}+\gamma\nabla(\gamma^{-1}J)w_1J^{-1}\right]^s\right)
\\
&\quad + \left(\left[(I-\gamma J^{-1})\nabla\tilde{v}_2-\nabla(\gamma J^{-1})\tilde{v}_2+\nabla\tilde{v}_2(J^{-1}-I)\right]^s,\nablasym w_1\right)\\
&=:(\nablasym \hat{v}_2:\nablasym w_1)+R_2.
\end{aligned} \end{align}
\begin{proof}[Estimate of $R_2$]
	By the definition of $J$ it is straightforward to see that
	\begin{align*}
	|J^{-1}| &\le C(1+|\nabla \gamma|),\\
	|J^{-1}-I| +|J-I|+|\gamma J^{-1}-I| &\le C(|\gamma-1|+|\nabla \gamma|).
	\end{align*}
	By H\"older's inequality we get $\||\nabla \tilde{v}_2||\nabla w_1|\|_{6/5} \le \|\nabla \tilde{v}_2\|_{3} \|\nabla w_1\|_{2}$ and thus for $p=6/5$
	\begin{align*}
	\left((\nabla\tilde{v}_2J^{-1})^s, \left[\nabla w_1 (J^{-1}-I) +(J-I)\nabla w_1J^{-1}\right]^s\right)+\left(\left[(I-\gamma J^{-1})\nabla\tilde{v}_2+\nabla\tilde{v}_2(J^{-1}-I)\right]^s,\nablasym w_1\right)\\
	\le C\|\nabla \gamma\|_{3p'} ~\||\nabla\tilde{v}_2|~ |\nabla w_1|\|_{p}
	\le \|\eta\|_{2,2}( \|\nabla v_2\|_{s} \|\nabla w_1\|)\\
	\le C_\epsilon\|\eta\|_{2,2}^2\| v_2\|^2_{W^{1,s}(\Omega_2)}+\epsilon \|\nabla w_1\|^2
	\end{align*}
	Furthermore as in the proof of Lemma \ref{liste} we get for $p \in (3,s)$ (i.e. $p'\in (s',3/2)$)
	\begin{align*}
	&\left((\nabla\tilde{v}_2J^{-1})^s:(\nabla(\gamma^{-1}J)w_1J^{-1})^s\right)
	\le C \|(1+|\nabla \gamma|+|\nabla \gamma|^2+|\nabla \gamma|^3)|\nabla \tilde{v}_2|\|_{p}\||\nabla^2 \gamma| |w_1|\|_{p'}
	\\
	&\quad \le C\|\eta\|_{2,2} \|\nabla v_2\|_{1,s}\|w_1\|_{1,2}
	\le C_\epsilon\|\eta\|_{2,2}^2\| v_2\|^2_{W^{1,s}(\Omega_2)}+\epsilon\|w_1\|^2_{1,2}
	\end{align*}
	and
	\begin{equation*}
	\left((\nabla(\gamma J^{-1})\tilde{v}_2)^s,\nablasym w_1\right) \le C\|(|\nabla^2 \gamma|+|\nabla \gamma|)|\tilde{v}_2|\| ~\|\nabla w_1\| \le 
	C_\epsilon\| v_2\|_{W^{1,s}(\Omega_2)}^2\|\eta\|_{2,2}^2+\epsilon\|w_1\|_{1,2}^2
	\end{equation*}
	In conclusion
	\begin{equation}
	\label{ineq21}
	|R_2| \le C_\epsilon\|\eta\|_{2,2}^2\| v_2\|^2_{W^{1,s}(\Omega_2)}+\epsilon\|w_1\|^2_{1,2}
	\end{equation}
\end{proof}
Next by chain rule and \eqref{eq:what} we get
\begin{align} \begin{aligned}
\label{ineq3}
\skp{[\nabla v_2]v_2}{w_2}_{\eta_2}&=
([\nabla \tilde{v}_2]\gamma J^{-1}\tilde{v}_2, \gamma^{-1}Jw_1)
= ([\nabla \tilde{v}_2]\gamma J^{-1}\tilde{v}_2, \gamma^{-1}Jw_1)
\\
&=([\nabla \tilde{v}_2]\hat{v}_2, w_1)+([\nabla \tilde{v}_2]\hat{v}_2, (\gamma^{-1}J-I)w_1)
\\
&=([\nabla\hat{v}_2]\hat{v}_2, w_1)+([\nabla((I-\gamma^{-1}J)\tilde{v_2})]\hat{v}_2, w_1)+([\nabla \tilde{v}_2]\hat{v}_2, (\gamma^{-1}J-I)w_1)
\\
&:=  ([\nabla\hat{v}_2]\hat{v}_2, w_1)+R_3
\end{aligned} \end{align}
\begin{proof}[Estimate on $R_3$]
	With similar estiamtes as above we can conclude
	\begin{align*}
	&([\nabla\tilde{v}_2]\hat{v}_2, (\gamma^{-1}J-I)w_1)
	\le C\|\hat{v}_2\|_\infty \|(\gamma^{-1}-J^t)\nabla \tilde{v}_2\| \|w_1\|\\
	&\quad \le 
	C\|\eta\|_{2,2}\| v_2\|^2_{W^{1,s}(\Omega_2)}\|w_1\| \le C\| v_2\|^2_{W^{1,s}(\Omega_2)}(\|\eta\|_{2,2}^2+\|w_1\|^2)
	\end{align*}
	Additionally 
	\begin{align*}
	([\nabla((I-\gamma^{-1}J)\tilde{v_2})]\hat{v}_2, w_1)
	&\le C\|\hat{v}_2\|_{L^\infty} \|w_1\|(\|\tilde{v}_2\|_\infty \|\nabla(\gamma J^{-1})\|+\|(I-\gamma J^{-1})\nabla \tilde{v}_2\|\\
	&\le C\| v_2\|^2_{W^{1,s}(\Omega_2)}(\|\eta\|_{2,2}^2+\|w_1\|^2)
	\end{align*}
	Thus 
	\begin{equation} 
	\label{ineq31}
	|R_3| \le  C\| v_2\|^2_{W^{1,s}(\Omega_2)}(\|\eta\|_{2,2}^2+\|w_1\|^2)
	\end{equation}
\end{proof}
Lastly by transformation rule and \eqref{eq:what}
\begin{equation}
\label{ineq4}
\skp{f_2}{w_2}_{\eta_2}=(\tilde{f}_2, \gamma \tilde{w}_2)=(\tilde{f}_2,w_1)-(\tilde{f}_2, (I-J)w_1)\equiv(\tilde{f}_2 ,w_1)-R_4.
\end{equation}
We find for all $p\in (1,\infty)$
\begin{equation}
\label{ineq41}
R_4 \le C\|(1-\gamma)\|_{p'} \||\tilde{f}_2||w_1|\|_p \le  C_{\epsilon}\|f_2\|^2_{L^2(\Omega_2)} \|\eta\|_{2,2}^2+\epsilon\|w_1\|_{1,2}^2
\end{equation}

Adding \eqref{ineq1}, \eqref{ineq2}, \eqref{ineq3}, \eqref{ineq4} and integrating over $(0,t)$ we get
\begin{align} \begin{aligned}
\label{eqfinal}
&\int_0^t( \partial_t\hat{v}_2+[\nabla\hat{v}_2]\hat{v}_2, w_1)+(\nablasym\hat{v}_2,\nablasym w_1)-(\tilde{f}_2, w_1)~dt\\
&\quad =\int_0^t\skp{\partial_t v_2+[\nabla v_2]v_2}{w_2}_{\eta_2}+\skp{\nablasym v_2}{\nablasym w_2}_{\eta_2}-\skp{f_2}{w_2} ~dt+R,
\end{aligned} \end{align}
where $R=\int_0^t R_1+R_2+R_3+R_4~dt$. By \eqref{ineq11}, \eqref{ineq21}, \eqref{ineq31}, \eqref{ineq41} we get
\begin{align} \begin{aligned}
\label{eqR}
|R| \le \int_0^t h_1(t)(\|\eta\|_{2,2}^2+\|\partial_t \eta\|^2_2+\|w_1\|^2_2)+\epsilon \|w_1\|_{1,2}^2~dt,\\
h_1(t)=C_\epsilon(\| v_2\|^2_{W^{1,s}(\Omega_2)}+\|\partial_t v_2\|_{W^{-1,s}(\Omega_2)}^2+\|f_2\|^2_{L^2(\Omega_2)}) \in L^1([0,T]).
\end{aligned} \end{align}

We can now estimate the differences of the solutions, namely we estimate 
\begin{align*}
I:&=\frac{1}{2} \|w_1\|^2+\frac{1}{2}\left( \|\partial_t \eta\|^2+\|\Delta \eta\|^2\right)+\int_0^t\|\nablasym w_1\|^2dt
\\
&=
\frac{1}{2} \|v_1\|^2+\frac{1}{2}\left( \|\partial_t \eta_1\|^2+\|\Delta \eta_1\|^2\right)
\\
&\quad
-(v_1(t),\hat{v}_2(t))- (\partial_t \eta_1,\partial_t \eta_2)-(\Delta \eta_1, \Delta \eta_2)
\\
&\quad +\frac{1}{2}\left(\|\hat{v}_2\|^2+\|\eta_2\|^2+\|\Delta\eta_2\|^2\right)
\\
&\quad
+\int_0^t \|\nablasym v_1\|^2 - (\nablasym v_1,\nablasym \hat{v}_2) -(\nablasym \hat{v}_2, \nablasym w_1)~dt.
\end{align*}
The energy inequality for $(v_1, \eta_1)$ gives
\begin{align} \begin{aligned}
\label{eqJ}
I &\le\frac{1}{2} (\|v_{1,0}\|^2+\|\eta_{1,0}^*\|^2+\|\Delta \eta_{1,0}\|^2)+\int_0^t (f_1, v_1)_{\eta_1}+(g_1,\partial_t \eta_1)~dt
\\
&\quad -( v_1(t),\hat{v}_2(t))-\int_0^t (\nablasym v_1,\nablasym \hat{v}_2)dt-(\partial_t \eta_1,\partial_t \eta_2)-(\Delta \eta_1, \Delta \eta_2)
\\
&\quad
+\frac{1}{2}(\|\hat{v}_2\|^2+\|\partial_t\eta_2\|^2+\|\Delta\eta_2\|^2)-\int_0^t (\nablasym \hat{v}_2, \nablasym w_1)~dt
\end{aligned} 
\end{align}
By \eqref{eqv22} and \eqref{eqfinal} we get
\begin{align*}
-\int_{0}^t(\nablasym \hat{v}_2,\nablasym w_1)~dt&= \int_0^t (\partial_t \hat{v}_2+([\nabla\hat{v}_2]\hat{v}_2, w_1)-(\tilde{f}_2,w_1)~dt
+(\partial_t \eta_2,\partial_t \eta_\delta)-(\eta_{2,0}^*,\eta_0^*)
\\
&\quad -\int_0^t\int_{\omega}\partial_{t}\eta_2\partial_t^2\eta_{\delta}-(\Delta \eta_2, \Delta \partial_t \eta_\delta)+(g_2,\partial_t \eta_\delta)dt+K^1_\delta+R.
\end{align*}
Reynold's transport theorem and $\hat{v}_2(x,\eta_1(x))=\partial_t\eta_2(x)$ gives 
\begin{align*}
\int_0^t (\partial_t \hat{v}_2, w_1) ~dt&=
\int_0^t (\partial_t \hat{v}_2, v_1-\hat{v}_2)~ dt
\\
&=-\frac{1}{2}(\|\hat{v}_2\|^2-\|\hat{v}_{2,0}\|^2- (\partial_t \eta_1, (\partial_t \eta_2)^2))
+\int_0^t(\partial_t \hat{v}_2 , v_1)~ dt
\end{align*}
Inserting this calculation in \eqref{eqJ}  yields
\begin{align*}
I &\le \frac{1}{2}((\|v_{1,0}\|^2+\|\hat{v}_{2,0}\|^2)-(v_1(t),\hat{v}_2(t))+\int_0^t (v_1,\partial_t \hat{v}_2)- (\nablasym v_1,\nablasym \hat{v}_2)+(f_1,v_1)-(\tilde{f}_2, w_1)dt
\\
&\quad +\frac{1}{2}(\|\eta_{1,0}^*\|^2+\|\Delta \eta_{1,0}\|^2+\|\partial_t\eta_2(t)\|^2+\|\Delta\eta_2(t)\|^2)-(\partial_t \eta_1(t),\partial_t \eta_2(t))-(\Delta \eta_1(t), \Delta \eta_2(t))
\\
&\quad+(\partial_t \eta_2(t),\partial_t \eta_\delta(t))-(\eta_{2,0}^*,\eta_0^*)-\int_0^t(\partial_{t}\eta_2,\partial_t^2
\eta_{\delta})-(\Delta \eta_2,\Delta \partial_t \eta_{\delta})-(g_1, \partial_t \eta_1)+(g_2,\partial_t \eta_\delta)~dt
\\
&\quad +\int_0^t([\nabla\hat{v}_2]\hat{v}_2, w_1)+\frac{1}{2}(\partial_t \eta_1, (\partial_t \eta_2)^2)+K^1_\delta+R
\end{align*}
We denote the first line of the right hand side as $I_1$ the second and third line as $I_2$ and the fourth line as $I_3$. 
We calculate that
$$
\frac{1}{2}(\|v_{1,0}\|^2+\|\hat{v}_{2,0}\|^2)+\int_0^t (f_1,v_1)-(\tilde{f}_2, w_1)~dt=(v_{1,0},\hat{v}_{2,0})+\frac12\|v_{1,0}-\hat{v}_{2,0}\|^2+\int_0^t (f_1, \hat{v}_2)+(f_1-\tilde{f}_2,w_1)~dt.
$$
Thus
\begin{align*}
I_1&=(v_{1,0},\hat{v}_{2,0})-(v_1(t),\hat{v}_2(t))+\int_0^t (v_1,\partial_t \hat{v}_2)- (\nabla v_1,\nabla \hat{v}_2)+(f_1,\hat{v}_2)~dt
\\
&\quad +\frac{1}{2}\|v_{1,0}-\hat{v}_{2,0}\|^2_2+\int_0^t (f_1-\tilde{f}_2,w_1)~dt
\end{align*}
We write the first line as
\begin{align*}
&(v_{1,0},\hat{v}_{2,0})-(v_1(t),\hat{v}_2(t))+\int_0^t (v_1,\partial_t \hat{v}_2)- (\nabla v_1,\nabla \hat{v}_2)+(f_1,\hat{v}_2)~dt
\\
&\quad = (v_{1,0},\hat{v}_{2,0})-(v_1(t),\hat{v}_{2,\delta}(t))+\int_0^t (v_1,\partial_t\hat{v}_{2,\delta})- (\nabla v_1,\nabla \hat{v}_{2,\delta}) )+(f_1,\hat{v}_{2,\delta})~dt+K_{2,\delta},
\end{align*}
with
$$K_{2,\delta}=-(v_1,\hat{v}_2-\hat{v}_{2,\delta})+\int_0^t (v_1,\partial_t \hat{v}_2-\partial_t\hat{v}_{2,\delta})- (\nabla v_1,\nabla \hat{v}_2-\nabla \hat{v}_{2,\delta})+(f_1,\hat{v}_2-\hat{v}_{2,\delta})~dt,$$
which converges to zero for $\delta \to 0$ by Lemma \ref{liste3}.
We divide $I_2$ into the parts that depend solely on $\eta_2$ and the rest:
\begin{align*}
I_2&=\frac{1}{2}(\|\Delta \eta_2\|^2-\|\partial_t \eta_2\|^2) +\|\eta_{2,0}^*\|^2+\int_0^t (\partial_t \eta_2,\partial_t^2 \eta_{2,\delta})-(\Delta \eta_2,\partial_t \Delta \eta_{2,\delta})~dt
\\
&\quad +\frac{1}{2}(\|\eta_{1,0}^*\|^2+\|\Delta \eta_{1,0}\|^2)-(\partial_t \eta_1(t),\partial_t \eta_2(t))-(\Delta \eta_1(t), \Delta \eta_2(t))
\\
&\quad+(\partial_t \eta_2(t),\partial_t \eta_{1,\delta}(t))-(\eta_{2,0}^*,\eta_{1,0}^*)-\int_0^t(\partial_{t}\eta_2,\partial_t^2
\eta_{1,\delta})-(\Delta \eta_2,\Delta \partial_t \eta_{1,\delta})-(g_1, \partial_t \eta_1)+(g_2,\partial_t \eta_\delta)~dt
\end{align*}
%
We denote the first line by $I_{21}$ and find that
\begin{align*}
I_{21}&=\frac{1}{2}(\|\eta_{2,0}*\|^2+\|\Delta \eta_{2,0}\|^2)
\\
&\quad
+\frac{1}{2}(\|\eta_{2,0}^*\|^2-\|\partial_t\eta_2\|^2+\|\Delta \eta_2\|^2-\|\Delta\eta_{2,0}\|^2)
+\int_0^t (\partial_t \eta_2,\partial_t^2 \eta_{2,\delta})-(\Delta \eta_2,\partial_t \Delta \eta_{2,\delta})~dt
\\
&=: \frac{1}{2}(\|\eta_{2,0}^*\|^2+\|\Delta \eta_{2,0}\|^2)+K_{3,\delta}
\end{align*}
where $K_{3,\delta} \to 0$ for $\delta \to 0$ by Lemma~\ref{lemma}.

Collecting the above we arrive at
\begin{align} \begin{aligned}
\label{I100}
I &\le(v_{1,0},\hat{v}_{2,0})-(v_1(t),\hat{v}_{2,\delta}(t))+\int_0^t (v_1,\partial_t\hat{v}_{2,\delta})- (\nablasym v_1,\nablasym \hat{v}_{2,\delta}) )+(f_1,\hat{v}_{2,\delta})~dt
\\
&\quad +\frac{1}{2}(\|\eta_{1,0}^*\|^2+\|\eta_{2,0}^*\|^2+\|\Delta \eta_{1,0}\|^2+\|\Delta \eta_{2,0}\|^2)+(\partial_t \eta_{1,\delta}(t)-\partial_t \eta_1(t), \partial_t \eta_2(t))-(\Delta \eta_1(t), \Delta \eta_2(t))
\\
&\quad -(\eta_{2,0}^*,\eta_{1,0}^*)-\int_0^t(\partial_{t}\eta_2,\partial_t^2\eta_{1,\delta})-(\Delta \eta_2,\Delta \partial_t \eta_{1,\delta})-(g_1,\partial_t \eta_1)+(g_2, \partial_t \eta_\delta)~dt
\\
&\quad
+\frac{1}{2}\|v_{1,0}-\hat{v}_{2,0}\|^2_2+\int_0^t (f_1-\tilde{f}_2,w_1)~dt+I_3+K_{2,\delta}+K_{3,\delta}.
\end{aligned} \end{align}
Now we use the equation vor $(v_1, \eta_1)$ and test it with $\mol$:
\begin{align} \begin{aligned}
\label{help}
&(v_{1,0},\hat{v}_{2,0})-(v_1(t),\hat{v}_{2,\delta}(t))+\int_0^t (v_1,\partial_t\hat{v}_{2,\delta})- (\nablasym v_1,\nablasym \hat{v}_{2,\delta})+(f_1,\hat{v}_{2,\delta})~dt
\\
&\quad =
-\int_0^t(v_1 \otimes v_1, \nabla \hat{v}_{2,\delta})~dt+(\partial_t \eta_1(t) ,\partial_t \eta_{2,\delta}(t))-(\eta_{1,0}^*,\eta_{2,0}^*)
\\
&\qquad -\int_0^t (\partial_t \eta_1,\partial_t^2 \eta_{2,\delta})-(\Delta \eta_1,\Delta \partial_t \eta_{2,\delta})+(g_1,\partial_t \eta_{2,\delta})~dt
\end{aligned} 
\end{align}
Note that 
\begin{align*}
&\frac{1}{2}(\|\eta_{1,0}^*\|^2+\|\eta_{2,0}^*\|^2+\|\Delta \eta_{1,0}\|^2+\|\Delta \eta_{2,0}\|^2)\\
&\quad =(\eta_{1,0}^*,\eta_{2,0}^*)+(\Delta \eta_{1,0},\Delta \eta_{2,0})+\frac{1}{2}(\|\eta_{1,0}^*-\eta_{2,0}^*\|^2+\|\Delta \eta_{1,0}-\Delta \eta_{2,0}\|^2)
\end{align*}
and 
\begin{align*}
(g_1,\partial_t \eta_1)-(g_2, \partial_t \eta_\delta)&=(g_1, \partial_t \eta_1)-(g_2,\partial_t \eta)+(g_2,\partial_t \eta -\partial_t \eta_\delta)
\\
&=(g_1,\partial_t \eta_2)+(g_1-g_2,\partial_t \eta)+(g_2,\partial_t \eta -\partial_t \eta_\delta).
\end{align*}
This gives
\begin{align} \begin{aligned}
\label{I200}
I \le -\int_0^t([\nabla \hat{v}_2]v_1,v_1)~dt+\frac{1}{2}(\|v_{1,0}-\hat{v}_{2,0}\|^2_2+\|\eta_{1,0}^*-\eta_{2,0}^*\|^2+\|\Delta \eta_{1,0}-\Delta \eta_{2,0}\|^2)\\
+\int_0^t (f_1-\tilde{f}_2,w_1)+(g_1-g_2,\partial_t\eta)~dt+K_{2,\delta}+K_{3,\delta}+K_{4,\delta}+I_3
\end{aligned} \end{align}
where
\begin{align*}
K_{4,\delta}&=(\Delta \eta_{1,0},\Delta \eta_{2,0})-(\Delta \eta_1(t),\Delta\eta_2(t))+\int_0^t (\Delta \eta_2,\partial_t \Delta \eta_{1,\delta})+(\Delta \eta_1, \partial_t\Delta\eta_{2,\delta})~dt
+(\partial_t \eta_{1,\delta}-\partial_t \eta_1,\partial_t \eta_2)
\\
&\quad +(\partial_t \eta_1(t),\partial_t \eta_{2,\delta}(t))-(\eta_{1,0}^*,\eta_{2,0}^*)-\int_0^t (\partial_t \eta_2, \partial_t^2\eta_{1,\delta})+(\partial_t \eta_1, \partial_t^2 \eta_{2,\delta})~dt
\\
&\quad+\int_0^t (g_1,\partial_t \eta_2-\partial_t \eta_{2,\delta})+(g_2,\partial_t \eta-\partial_t \eta_\delta)~dt
\end{align*}
\begin{proof}[Proof that $K_{4,\delta}\to 0$]
	The first and third line of $K_{4,\delta} $ converge to $0$ again by Lemma $\ref{lemma}$.
	We write the second line as
	\begin{align*}
	&(\partial_t \eta_1(t),\partial_t \eta_{2,\delta}(t))-(\eta_{1,0}^*,\eta_{2,0}^*)-\int_0^t (\partial_t \eta_2, \partial_t^2\eta_{1,\delta})+(\partial_t \eta_1, \partial_t^2 \eta_{2,\delta})~dt
	\\
	&\quad =(\partial_t \eta_1(t),\partial_t \eta_{2,\delta}(t)-\partial_t \eta_2(t))
	+(\partial_t \eta_1(t),\partial_t \eta_2(t))-(\eta_{1,0}^*,\eta_{2,0}^*)-\int_0^t (\partial_t \eta_2, \partial_t^2\eta_{1,\delta})+(\partial_t \eta_1, \partial_t^2 \eta_{2,\delta})~dt,
	\end{align*}
	which also converges to $0$ for $\delta \to 0$ by Lemma~\ref{lemma}. Thus $K_{4,\delta} \to 0$ for $\delta \to 0$.
\end{proof}
We continue by writing
\begin{align*}-\int_0^t(v_1 \otimes v_1, \nabla \hat{v}_{2,\delta})~dt
&=-\int_0^t(v_1 \otimes v_1, \nabla \hat{v}_{2})~dt+\int_0^t (v_1 \otimes v_1,\nabla \hat{v}_{2}-\nabla \hat{v}_{2,\delta})~dt
\\
&=:-\int_0^t(v_1 \otimes v_1,\nabla \hat{v}_{2})~dt+K_{5,\delta},
\end{align*}
where $K_{5,\delta} \to 0$ by Lemma \ref{liste3}. Inserting this and the definition of $I_3$ in \eqref{I200} finally yields
\begin{align} \begin{aligned}
\label{I300}
I &\le \int_0^t- (v_1 \otimes v_1,\nabla \hat{v}_{2})
+([\nabla\hat{v}_2]\hat{v}_2, w_1)+\frac{1}{2}(\partial_t \eta_1, (\partial_t \eta_2)^2)
~dt\\
&+\frac{1}{2}(\|v_{1,0}-\hat{v}_{2,0}\|^2_2+\|\eta_{1,0}^*-\eta_{2,0}^*\|^2+\|\Delta \eta_{1,0}-\Delta \eta_{2,0}\|^2)
+\int_0^t (f_1-\tilde{f}_2,w_1)+(g_1-g_2,\partial_t\eta)~dt\\
&+R+K_{1,\delta}+K_{2,\delta}+K_{3,\delta}+K_{4,\delta}+K_{5,\delta}
\end{aligned} \end{align}
The first line can be estimated as follows. As $\di v_1=0$ we get by Gau\ss   integral formula
$$([\nabla \hat{v}_2]v_1,\hat{v}_2)_{\eta_1}=\frac{1}{2}((\partial_t\eta_2)^2),\partial_t\eta_1)_{\omega}$$
Hence 
$$(v_1 \otimes v_1, \nabla\hat{v}_2)=([\nabla\hat{v}_2]v_1,v_1-\hat{v}_2)+([\nabla\hat{v}_2]v_1,\hat{v}_2)=([\hat{v}_2]v_1,w_1)+\frac{1}{2}((\partial_t\eta_2)^2),\partial_t\eta_1)$$
Thus we get
$$
\int_0^t- (v_1 \otimes v_1,\nabla \hat{v}_{2})
+([\nabla\hat{v}_2]\hat{v}_2, w_1)+\frac{1}{2}(\partial_t \eta_1, (\partial_t \eta_2)^2)~dt=-\int_0^t ([\nabla \hat{v}_2]w_1,w_1)~dt
$$
We can estimate this term the same way as~\eqref{con2} in Lemma \ref{liste3} by replacing $v_1$ by $w_1$ and $\partial_t\eta_1$ by $\partial_t \eta$. We find
\begin{align*}
([\nabla \hat{v}_2]w_1,w_1) &  \le C_\epsilon (\|v_2\|_{W^{1,s}(\Omega_2)}+1)^2(\|\eta_1\|_{1,\infty}+\|\eta_2\|_{1,\infty}+1)^2(\|w_1\|^2+\|\partial_t \eta\|^2)+\epsilon \| w_1\|_{1,2}^2
\end{align*}
Thus
\begin{align*}
\int_0^t ([\nabla \hat{v}_2]w_1,w_1)~dt &\le \int_0^t h_2(t)(\|\partial_t \eta\|^2+\|w_1\|^2)+\epsilon \|\nabla w_1\|_{2}^2~dt,\\
&h_2(t)=(\|v_2\|_{W^{1,s}(\Omega_2)}+1)^2(\|\eta_1\|_{1,\infty}+\|\eta_2\|_{1,\infty}+1)^2.
\end{align*} 
As $v_2 \in L^r(0,T;W^{1,s}(\Omega_2))$ ($r>2$) and $\eta_1, \eta_2 \in L^p(0,T;W^{1,\infty}(\omega))$ for all $p \in [1,\infty)$ (by Theorem~\ref{thm:boris} and interpolation) we get $h_2 \in L^1([0,T])$.

Thus recalling the estimate on $R$ \eqref{eqR} we get
	\begin{align*}
	\label{eqtR}
	\int_0^t ([\nabla \hat{v}_2]w_1,w_1)~dt+R &\le \int_0^t h(t)(\|\eta\|_{2,2}^2+\|\partial_t \eta\|^2+\|w_1\|^2)+\epsilon \|\nabla w_1\|_{2}^2~dt,\\
	h&=h_1+h_2 \in L^1([0,T]).
	\end{align*}
	Since $K_{i,\delta} \to 0$ for $\delta \to 0$ ($i=1,\ldots,5$)  the last estimate leads to
	\begin{align*}
	&\frac{1}{2}( \|w_1\|^2+\|\partial_t \eta\|^2+\|\Delta \eta\|^2)+\int_0^t\|\nablasym w_1\|^2dt\\
	&\le \frac{1}{2}(\|v_{1,0}-\hat{v}_{2,0}\|^2_2+\|\eta_{1,0}^*-\eta_{2,0}^*\|^2+\|\Delta \eta_{1,0}-\Delta \eta_{2,0}\|^2)+\int_0^t\|f_1-\tilde{f}_2\|_2+ \|g_1-g_2\|_2^2~dt\\
	&\quad +\int_0^t h(t)(\|\eta\|_{2,2}^2+\|\partial_t \eta\|^2+\|w_1\|^2)+\epsilon \|w_1\|_{1,2}^2~dt
	\end{align*}
As $\eta$ is $0$ on the boundary $\|\eta\|_{2,2} \sim \|\Delta \eta\| $. Korn's inequality and the $0$ trace of $w_1$ on $B_c$ implies that $\norm{w_1}_{1,2}\sim \norm{\nablasym w_1}_2 $. 
Hence choosing $\epsilon<1$ small enough we can apply Gronwall's Lemma. this implies a stability estimate in terms of $w_1$. In order to change to $v_1-\tilde{v}_2$ one uses
\[
\norm{w_1}_2\leq \norm{v_1-\tilde{v}_2}_2+\norm{\tilde{v}_2-\hat{v}_2} \leq \norm{v_1-\tilde{v}_2}_2+C\|\eta\|_{1,2};
\]
the estimate on the gradients is analogous.
This finishes the proof of Theorem~\ref{the:2}.
\begin{proof}[Proof of Theorem~\ref{the:1}]
	Let $(v_1,\pi_1,\eta_1)$ be a weak solution (with $\eta>0$) on $[0,T]$ for any $T>0$. Then the stability estimate implies that $\|w_1\|=\|\partial_t \eta\|=\|\eta\|_{2,2}=0$ a.e. in $[0,T]$. As $\eta=\eta_1-\eta_2=0$ we have $\Omega_1=\Omega_2$ and in particular the transformation $\psi$ is the identitiy, $\gamma=1$, $J=\mathbb{I}$. Thus $\hat{v}_2=v_2$ and $w_1=0$ gives $v_1=v_2$. This proves Theorem \ref{the:1}. 
\end{proof}

\bibliographystyle{plain}


\end{document}